\documentclass{article}%
\usepackage{amsmath}
\usepackage{amsfonts}
\usepackage{amssymb}
\usepackage{graphicx}%
\providecommand{\U}[1]{\protect\rule{.1in}{.1in}}
\newtheorem{theorem}{Theorem}

\newtheorem{lemma}[theorem]{Lemma}

\newtheorem{proposition}[theorem]{Proposition}

\newenvironment{proof}[1][Proof]{\noindent\textbf{#1.} }{\ \rule{0.5em}{0.5em}}
\begin{document}

\title{Optimal Robustness Results for Some Bayesian Procedures and the Relationship
to Prior-Data Conflict}
\author{Luai Al-Labadi and Michael Evans\\Department of Statistics\\University of Toronto}
\date{}
\maketitle

\begin{abstract}
The robustness to the prior of Bayesian inference procedures based on a
measure of statistical evidence are considered. These inferences are shown to
have optimal properties with respect to robustness. Furthermore, a connection
between robustness and prior-data conflict is established. In particular, the
inferences are shown to be effectively robust when the choice of prior does
not lead to prior-data conflict. When there is prior-data conflict, however,
robustness may fail to hold.

\end{abstract}

\section{Introduction}

Robustness to the choice of the prior is an issue of considerable importance
in a Bayesian statistical analysis. If an inference is very sensitive to the
choice of the prior, then this could be viewed as either a negative for the
inference method being used or for the choice of prior. In this paper it is
shown that certain inferences are in a sense optimally robust to the choice of
the prior. Furthermore, when the sensitivity of the inferences to the prior is
measured quantitatively, it is shown that there is an intimate connection
between the effective robustness of the inferences and whether or not there is
prior-data conflict. So by choice of the inferential methodology and the
avoidance of prior-data conflict, robustness of the inferences to the choice
of prior is achieved.

The basic ingredients for a statistical analysis are taken here to be the data
$x,$ a statistical model $\{f_{\theta}:\theta\in\Theta\},$ where each
$f_{\theta}$ is a probability density with respect to volume measure $\mu$ on
the sample space $\mathcal{X},$ and a proper prior density $\pi$ with respect
to volume measure $\nu$ on $\Theta.$ Note that volume measure on a discrete
set is taken to be counting measure. Furthermore, suppose that interest is in
making inferences about the quantity $\psi=\Psi(\theta)$ where $\Psi
:\Theta\rightarrow\Psi$ is onto and we don't distinguish between the function
and its range to save notation.

Let $\pi_{\Psi}(\cdot\,|\,x)$ and $\pi_{\Psi}$ denote the posterior and prior
densities of $\psi$ where these are both taken with respect to support measure
$\nu_{\Psi}$ on $\Psi.$ It follows that, under smoothness assumptions,
$\pi_{\Psi}(\psi)=\int_{\Psi^{-1}\{\psi\}}\pi(\theta)J_{\Psi}(\theta
)\,\nu_{\Psi^{-1}\{\psi\}}(d\theta)$ where $J_{\Psi}(\theta)=(\det
(d\Psi(\theta)\circ(d\Psi(\theta))^{t})^{-1/2},d\Psi$ is the differential of
$\Psi$ and $\nu_{\Psi^{-1}\{\psi\}}$ is volume measure on $\Psi^{-1}\{\psi\}.$
Also, $\pi_{\Psi}(\psi\,|\,x)=\int_{\Psi^{-1}\{\psi\}}\pi(\theta
\,|\,x)J_{\Psi}(\theta)\,\nu_{\Psi^{-1}\{\psi\}}(d\theta)$ where $\pi
(\theta\,|\,x)=\pi(\theta)f_{\theta}(x)/m(x),$ with $m(x)=\int_{\Theta}%
\pi(\theta)f_{\theta}(x)\,\nu(d\theta),$ is the posterior density of $\theta$
with respect to $\nu.$ Note that $m$ is the prior predictive density of the
data with respect to $\mu.$ The conditional prior of $\theta$ given $\psi
=\Psi(\theta)$ has density $\pi(\theta\,|\,\psi)=\pi(\theta)J_{\Psi}%
(\theta)/\pi_{\Psi}(\psi)$ with respect to $\nu_{\Psi^{-1}\{\psi\}}$ on the
set $\Psi^{-1}\{\psi\}.$ The conditional prior predictive density of $x$ is
then given by $m(x\,|\,\psi)=\int_{\Psi^{-1}\{\psi\}}\pi(\theta\,|\,\psi
)f_{\theta}(x)\,\nu_{\Psi^{-1}\{\psi\}}(d\theta).$ A simple argument, see
Baskurt and Evans (2013), gives the Savage-Dickey ratio result that
\begin{equation}
\frac{\pi_{\Psi}(\psi\,|\,x)}{\pi_{\Psi}(\psi)}=\frac{m(x\,|\,\psi)}{m(x)},
\label{savagedickey}%
\end{equation}
which has some use in the developments here.

Robustness to the prior has been considered by many authors and there are a
number of different approaches. Many discussions are concerned with
determining the range of values that some characteristic of interest takes
when the prior is allowed to vary over some class. Berger (1990, 1994) contain
broad reviews of work on this topic and Rios Insua and Ruggeri (2000) is a
collection of papers by key contributors. Dey and Birmiwal (1994) considers
global robustness measures based upon measures of distance from the posterior distribution.

The approach taken here is to study robustness to the prior for relative
belief inferences for $\psi$ rather than all possible inferences. Relative
belief inferences are based on the relative belief ratio defined
\begin{equation}
RB_{\Psi}(\psi\,|\,x)=\lim_{\delta\rightarrow0}\frac{\Pi_{\Psi}\left(
N_{\delta}(\psi)\,|\,x\right)  }{\Pi_{\Psi}\left(  N_{\delta}(\psi)\right)  }
\label{rbdef}%
\end{equation}
whenever this limit exists for a sequence of neighborhoods $N_{\delta}(\psi)$
of $\psi$ converging nicely to $\psi$ (see Rudin (1974) for the definition of
'converging nicely'). Under mild regularity conditions the limit exists and is
given by $RB_{\Psi}(\psi\,|\,x)=\pi_{\Psi}(\psi\,|\,x)/\pi_{\Psi}(\psi).$
Since $RB_{\Psi}(\psi\,|\,x)$ measures the change in belief that $\psi$ is the
true value it is a measure of evidence. Here $RB_{\Psi}(\psi\,|\,x)>1$ means
that there is evidence in favor of $\psi$ being the true value, as belief in
$\psi$ has increased after seeing the data, and $RB_{\Psi}(\psi\,|\,x)<1$
means that there is evidence against $\psi$ being the true value, as belief in
$\psi$ has decreased after seeing the data. Section 2 provides some more
details concerning relative belief inferences for both estimation and
hypothesis assessment but also see Baskurt and Evans (2013). Results in
Section 3 establish that these inferences have optimal robustness properties
when the marginal prior for $\psi$ is allowed to vary over all possibilities
in the class of $\epsilon$-contaminated priors. This generalizes results found
in Wasserman (1989), Ruggeri and Wasserman (1993) and de la Horra and
Fernandez (1994). Furthermore, an ambiguity concerning the interpretation of
the results is resolved. As such this provides further justifications for
these inferences.

While inferences may be optimally robust, this does not imply that they are in
fact robust. In Section 4 quantitative measures of the sensitivity of relative
belief inferences to both the marginal prior of $\psi$ and the conditional
prior for $\theta$ given $\Psi(\theta)=\psi$ are derived. In Section 5 it is
shown that these inferences are indeed robust when the base prior $\pi$ does
not suffer from prior-data conflict. This adds weight to arguments concerning
the importance of checking for prior-data conflict before reporting
inferences, as prior-data conflict can imply sensitivity of the inferences to
the choice of the prior. Prior-data conflict is interpreted as the true value
lying in the tails of the prior and consistent methods have been developed for
assessing this in Evans and Moshonov (2006) and Evans and Jang (2011a).
Methodology for modifying a prior when prior-data conflict is encountered,
through the selection of a prior weakly informative with respect to the base
prior, is developed in Evans and Jang (2011b).

\section{Relative Belief Inferences}

When $RB_{\Psi}(\psi\,|\,x)>1$ this is the factor by which prior belief in the
truth of $\psi$ has increased after seeing the data. Clearly the bigger
$RB_{\Psi}(\psi\,|\,x)$ the more evidence there is in favor of $\psi$ while,
when $RB_{\Psi}(\psi\,|\,x)<1,$ the smaller $RB_{\Psi}(\psi\,|\,x)$ is the
more evidence there is against $\psi.$ This leads to a total preference
ordering on $\Psi,$ namely, $\psi_{1}$ is not preferred to $\psi_{2}$ whenever
$RB_{\Psi}(\psi_{1}\,|\,x)\leq RB_{\Psi}(\psi_{2}\,|\,x)$ since there is at
least as much evidence for $\psi_{2}$ as there is for $\psi_{1}.$ This in turn
leads to unambiguous solutions to inference problems.

The best estimate of $\psi$ is the value for which the evidence is greatest,
namely,%
\[
\psi(x)=\arg\sup RB_{\Psi}(\psi\,|\,x).
\]
Associated with this estimate is a $\gamma$-relative belief credible region
$C_{\Psi,\gamma}(x)=\{\psi:RB_{\Psi}(\psi\,|\,x)\geq c_{\Psi,\gamma}(x)\}$
where $c_{\Psi,\gamma}(x)=\inf\{k:\Pi_{\Psi}(RB_{\Psi}(\psi\,|\,x)\leq
k\,|\,x)\geq1-\gamma\}.$ Notice that $\psi(x)\in C_{\Psi,\gamma}(x)$ for every
$\gamma\in\lbrack0,1]$ and so, for selected $\gamma,$ the size of
$C_{\Psi,\gamma}(x)$ can be taken as a measure of the accuracy of the estimate
$\psi(x).$ The interpretation of $RB_{\Psi}(\psi\,|\,x)$ as the evidence for
$\psi,$ forces the use of the sets $C_{\Psi,\gamma}(x)$ for our credible
regions. For if $\psi_{1}$ is in such a region and $RB_{\Psi}(\psi
_{2}\,|\,x)\geq RB_{\Psi}(\psi_{1}\,|\,x),$ then $\psi_{2}$ must be in the
region as well as there is at least as much evidence for $\psi_{2}$ as for
$\psi_{1}.$ Optimal properties for relative belief credible regions, in the
class of all credible regions, have been established in Evans, Guttman and
Swartz (2006) and Evans and Shakhatreh (2008) and optimal properties for
$\psi(x)$ are established in Evans and Jang (2011c).

For the assessment of the hypothesis $H_{0}:\Psi(\theta)=\psi_{0},$ the
evidence is given by $RB_{\Psi}(\psi_{0}\,|\,x).$ One problem that both the
relative belief ratio and the Bayes factor share as measures of evidence, is
that it is not clear how they should be calibrated. Certainly the bigger
$RB_{\Psi}(\psi_{0}\,|\,x)$ is than 1, the more evidence we have in favor of
$\psi_{0}$ while the smaller $RB_{\Psi}(\psi_{0}\,|\,x)$ is than 1, the more
evidence we have against $\psi_{0}.$ But what exactly does a value of
$RB_{\Psi}(\psi_{0}\,|\,x)=20$ mean? It would appear to be strong evidence in
favor of $\psi_{0}$ because beliefs have increased by a factor of 20 after
seeing the data. But what if other values of $\psi$ had even larger increases?
For example, the discussion in Baskurt and Evans (2013) of the
Jeffreys-Lindley paradox makes it clear that the value of a relative belief
ratio or a Bayes factor cannot always be interpreted as an indication of the
strength of the evidence.

The value $RB_{\Psi}(\psi_{0}\,|\,x)$ can be calibrated by comparing it to the
other possible values $RB_{\Psi}(\cdot\,|\,x)$ through its posterior
distribution. For example, one possible measure of the strength is
\begin{equation}
\Pi_{\Psi}(RB_{\Psi}(\psi\,|\,x)\leq RB_{\Psi}(\psi_{0}\,|\,x)\,|\,x)
\label{strength}%
\end{equation}
which is the posterior probability that the true value of $\psi$ has a
relative belief ratio no greater than that of the hypothesized value $\psi
_{0}.$ While (\ref{strength}) may look like a p-value, it has a very different
interpretation. For when $RB_{\Psi}(\psi_{0}\,|\,x)<1,$ so there is evidence
against $\psi_{0},$ then a small value for (\ref{strength}) indicates a large
posterior probability that the true value has a relative belief ratio greater
than $RB_{\Psi}(\psi_{0}\,|\,x)$ and so there is strong evidence against
$\psi_{0}.$ If $RB_{\Psi}(\psi_{0}\,|\,x)>1,$ so there is evidence in favor of
$\psi_{0},$ then a large value for (\ref{strength}) indicates a small
posterior probability that the true value has a relative belief ratio greater
than $RB_{\Psi}(\psi_{0}\,|\,x)$ and so there is strong evidence in favor of
$\psi_{0}.$ Notice that, in the set $\{\psi:RB_{\Psi}(\psi\,|\,x)\leq
RB_{\Psi}(\psi_{0}\,|\,x)\},$ the \textquotedblleft best\textquotedblright%
\ estimate of the true value is given by $\psi_{0}$ simply because the
evidence for this value is the largest in this set.

Various results have been established in Baskurt and Evans (2103) supporting
both $RB_{\Psi}(\psi_{0}\,|\,x)$, as the measure of the evidence for $H_{0}$,
and (\ref{strength}), as a measure of the strength of that evidence. For
example, the following simple inequalities are useful in assessing the
strength of the evidence, namely, $\Pi_{\Psi}(RB_{\Psi}(\psi\,|\,x)=RB_{\Psi
}(\psi_{0}\,|\,x)\,|\,x)\leq\Pi_{\Psi}(RB_{\Psi}(\psi\,|\,x)\leq RB_{\Psi
}(\psi_{0}\,|\,x)\,|\,x)\leq$\newline$RB_{\Psi}(\psi_{0}\,|\,x).$ So if
$RB_{\Psi}(\psi_{0}\,|\,x)>1$ and $\Pi_{\Psi}(\{RB_{\Psi}(\psi_{0}%
\,|\,x)\}\,|\,x)$ is large, there is strong evidence in favor of $\psi_{0}$
while, if $RB_{\Psi}(\psi_{0}\,|\,x)<1$ is very small, then there is
immediately strong evidence against $\psi_{0}.$ Also, in situations where
there are only a few possible values of $\psi,$ then $\Pi_{\Psi}(RB_{\Psi
}(\psi\,|\,x)=RB_{\Psi}(\psi_{0}\,|\,x)\,|\,x)$ can be a more appropriate
measure of strength.

When interest is in making inferences about $\psi=\Psi(\theta),$ it is
reasonable to ask how sensitive the relief belief approach is to the
ingredients given by the prior. This entails examining how dependent
$\psi(x),C_{\Psi,\gamma}(x),RB_{\Psi}(\psi_{0}\,|\,x)$ and $\Pi_{\Psi
}(RB_{\Psi}(\psi\,|\,x)\leq RB_{\Psi}(\psi_{0}\,|\,x)\,|\,x)$ are to changes
in the prior, as these four objects represent the essential relative belief inferences.

The full prior $\pi$ for $\theta$ can always be factored as $\pi(\theta
)=\pi_{\Psi}(\psi)\pi(\theta\,|\,\psi).$ In contrast to other discussions of
robustness with respect to the prior, the sensitivity of the inferences to
$\pi_{\Psi}$ and the sensitivity of the inferences to $\pi(\cdot\,|\,\psi)$
are considered separately, as this leads to more information concerning where
the lack of robustness arises when this occurs.

\section{Optimal Robustness With Respect to the Marginal Prior}

The result (\ref{savagedickey}) implies that $RB_{\Psi}(\psi
\,|\,x)=m(x\,|\,\psi)/m(x).$ From this it is immediate that $\psi(x)=\arg
\sup_{\psi}RB_{\Psi}(\psi\,|\,x)=\arg\sup_{\psi}m(x\,|\,\psi)$ and so the
relative belief estimate is optimally robust to $\pi_{\Psi}$ as the estimate
has no dependence on the marginal prior. Furthermore, $C_{\Psi,\gamma}(x)$ is
of the form $\{\psi:m(x\,|\,\psi)\geq k\}$ for some $k$ and so the form of
relative belief regions for $\psi$ is optimally robust to $\pi_{\Psi}.$ The
specific region chosen for the assessment of the accuracy of $\psi(x)$ depends
on the posterior and so is not independent of $\pi_{\Psi}.$ It is now proved
that $C_{\Psi,\gamma}(x)$ has an optimal robustness property among all
credible regions for $\psi.$

Consider $\epsilon$-contaminated priors for\ $\theta$ of the form
\begin{equation}
\Pi_{\epsilon}=\Pi(\cdot\,|\,\psi)\times\lbrack(1-\epsilon)\Pi_{\Psi}+\epsilon
Q], \label{prior1}%
\end{equation}
where $Q$ is a probability measure on $\Psi$ and $\Pi$ is the base prior as
described in the Introduction. Note that the conditional prior of $\theta$
given $\Psi(\theta)=\psi$ is fixed and independent of $\epsilon.$

To assess the robustness of the posterior content of a set $A\subset\Psi$ it
makes sense to look at $\delta(A)=\Pi_{\Psi}^{upper}\left(  A\,|\,x\right)
-\Pi_{\Psi}^{lower}\left(  A\,|\,x\right)  $ where $\Pi_{\Psi}^{upper}\left(
A\,|\,x\right)  =\sup_{Q}\Pi_{\Psi}^{\epsilon}\left(  A\,|\,x\right)  $ and
$\Pi_{\Psi}^{lower}\left(  A\,|\,x\right)  =\inf_{Q}\Pi_{\Psi}^{\epsilon
}\left(  A\,|\,x\right)  $ and the supremum/infimum is taken over all
probability measures on $\Psi.$ For this let $\epsilon^{\ast}=\epsilon
/(1-\epsilon)$ and $r(A)=\sup_{\psi\in A}RB_{\Psi}(\psi\,|\,x)=\sup_{\psi\in
A}m(x\,|\,\psi)/m(x),$ so $r(\Psi)=RB_{\Psi}(\psi_{LRSE}(x)\,|\,x)$ and always
one and only one of $r(A),r(A^{c})$ equals $r(\Psi).$

The following result is needed and a proof is provided in the Appendix.

\begin{lemma}
\label{huber1}(Huber (1973)) Let $Q$ denote a probability measure on $\Psi.$
For prior measure $\Pi_{\Psi}^{\epsilon}=(1-\epsilon)\Pi_{\Psi}+\epsilon Q$ on
$\Psi$ and $A\subset\Psi,$ (i) $\Pi_{\Psi}^{upper}\left(  A\,|\,x\right)
=(\Pi_{\Psi}\left(  A\,|\,x\right)  +\epsilon^{\ast}r(A))/(1+\epsilon^{\ast
}r(A)),$ (ii) $\Pi_{\Psi}^{lower}\left(  A\,|\,x\right)  =\Pi_{\Psi}\left(
A\,|\,x\right)  /(1+\epsilon^{\ast}r(A^{c})),$ (iii)%
\[
\delta(A)=\frac{\Pi_{\Psi}\left(  A\,|\,x\right)  \epsilon^{\ast}%
(r(A^{c})-r(A))}{(1+\epsilon^{\ast}r(A))(1+\epsilon^{\ast}r(A^{c}))}%
+\frac{\epsilon^{\ast}r(A)}{(1+\epsilon^{\ast}r(A))}%
\]
and (iv) $\delta(A^{c})=\delta(A).$
\end{lemma}

Let $\gamma^{\ast}(x)=\Pi_{\Psi}\left(  C_{\Psi,\gamma}(x)\,|\,x\right)  $ be
the exact posterior content of the $\gamma$-relative belief region. The
following result generalizes results found in Wasserman (1989) and de la Horra
and Fernandez (1994) who considered robustness to the prior of credible
regions for the full parameter $\theta$. In particular, this result applies to
arbitrary parameters $\psi=\Psi(\theta)$ and does not require continuity.

\begin{proposition}
\label{credregion}The following hold,\newline(i) among all sets $A\subset\Psi$
satisfying $\Pi_{\Psi}\left(  A\,|\,x\right)  \leq\gamma^{\ast}(x)$ and
$r(A)=r(\Psi),$ the set $C_{\Psi,\gamma}(x)$ minimizes $\delta(A),$%
\newline(ii) among all sets $A\subset\Psi$ satisfying $\Pi_{\Psi}\left(
A\,|\,x\right)  \geq\gamma^{\ast}(x)$ and $r(A^{c})=r(\Psi),$ the set
$C_{\Psi,1-\gamma^{\ast}(x)}^{c}(x)$ minimizes $\delta(A),$\newline(iii) when
$\gamma^{\ast}(x)=\gamma\geq1/2$ then, among all sets $A\subset\Psi$
satisfying $\Pi_{\Psi}\left(  A\,|\,x\right)  =\gamma,$ the set $C_{\Psi
,\gamma}(x)$ minimizes $\delta(A).$
\end{proposition}

\begin{proof}
(i) For any set $A$ with $r(A)=r(\Psi)$ then $r(A^{c})-r(A)=r(A^{c}%
)-r(\Psi)\leq0.$ Therefore,
\begin{align*}
\delta(A)  &  =\frac{\Pi_{\Psi}\left(  A\,|\,x\right)  \epsilon^{\ast}%
(r(A^{c})-r(\Psi))}{(1+\epsilon^{\ast}r(\Psi))(1+\epsilon^{\ast}r(A^{c}%
))}+\frac{\epsilon^{\ast}r(\Psi)}{1+\epsilon^{\ast}r(\Psi)}\\
&  \geq\frac{\Pi_{\Psi}\left(  C_{\Psi,\gamma}(x)\,|\,x\right)  \epsilon
^{\ast}(r(A^{c})-r(\Psi))}{(1+\epsilon^{\ast}r(\Psi))(1+\epsilon^{\ast}%
r(A^{c}))}+\frac{\epsilon^{\ast}r(\Psi)}{1+\epsilon^{\ast}r(\Psi)}.
\end{align*}
Now%
\begin{equation}
\frac{r(A^{c})-r(\Psi)}{1+\epsilon^{\ast}r(A^{c})} \label{eq1}%
\end{equation}
is increasing in $r(A^{c}),$ so we need to show that $A=C_{\Psi,\gamma}(x)$
minimizes $r(A^{c})$ among all $A$ satisfying $\Pi_{\Psi}\left(
A\,|\,x\right)  \leq\Pi_{\Psi}\left(  C_{\Psi,\gamma}(x)\,|\,x\right)  $ and
$r(A)=r(\Psi).$ Suppose that $r(A^{c})<r(C_{\Psi,\gamma}^{c}(x))$and let
$B=\{\psi:RB_{\Psi}(\psi\,|\,x)>r(A^{c})\}.\ $Note that $r(C_{\Psi,\gamma}%
^{c}(x))\leq\inf_{\psi\in C_{\Psi,\gamma}(x)}RB_{\Psi}(\psi\,|\,x)$ and so
$C_{\Psi,\gamma}(x)\subset B,$ which implies $\Pi_{\Psi}\left(
B\,|\,x\right)  >\Pi_{\Psi}\left(  C_{\Psi,\gamma}(x)\,|\,x\right)  $ with the
strictness of the inequality following from the definition of $C_{\Psi,\gamma
}(x).$ But also $B\subset A$ which contradicts $\Pi_{\Psi}\left(
A\,|\,x\right)  \leq\Pi_{\Psi}\left(  C_{\Psi,\gamma}(x)\,|\,x\right)  $ and
so we must have $r(A^{c})\geq r(C_{\Psi,\gamma}^{c}(x)).$ This establishes
that (\ref{eq1}) is minimized by $A=C_{\Psi,\gamma}(x).$\newline(ii) Now
consider all the sets $A$ with $r(A^{c})=r(\Psi).$ Since $\delta
(A)=\delta(A^{c}),$ it is equivalent to minimize $\delta(A^{c})$ among all
sets $A^{c}$ satisfying $\Pi_{\Psi}\left(  A^{c}\,|\,x\right)  \leq\Pi_{\Psi
}\left(  C_{\Psi,\gamma}^{c}(x)\,|\,x\right)  =1-\gamma^{\ast}(x)$ and
$r(A^{c})=r(\Psi).$ By part (i) this is minimized by taking $A^{c}%
=C_{\Psi,1-\gamma^{\ast}(x)}(x)$ and the result is proved.\newline(iii) The
solutions to the optimization problems in parts (i) and (ii), namely,
$C_{\Psi,\gamma}(x)$ and $C_{\Psi,1-\gamma}^{c}(x)$ respectively, both have
posterior content equal to $\gamma.$ As such one of these sets is the solution
to the optimization problem stated in (iii). We have that
\begin{align*}
&  \delta(C_{\Psi,\gamma}(x))-\delta(C_{\Psi,1-\gamma}^{c}(x))=\delta
(C_{\Psi,\gamma}(x))-\delta(C_{\Psi,1-\gamma}(x))\\
&  =\frac{\gamma\epsilon^{\ast}(r(C_{\Psi,\gamma}^{c}(x))-r(\Psi
))}{(1+\epsilon^{\ast}r(\Psi))(1+\epsilon r(C_{\Psi,\gamma}^{c}(x)))}%
-\frac{\gamma\epsilon^{\ast}(r(C_{\Psi,1-\gamma}^{c}(x))-r(\Psi))}%
{(1+\epsilon^{\ast}r(\Psi))(1+\epsilon r(C_{\Psi,1-\gamma}^{c}(x)))}\\
&  =\frac{\gamma\epsilon^{\ast}}{(1+\epsilon^{\ast}r(\Psi))}\left\{
\frac{r(C_{\Psi,\gamma}^{c}(x))-r(\Psi)}{1+\epsilon r(C_{\Psi,\gamma}^{c}%
(x))}-\frac{r(C_{\Psi,1-\gamma}^{c}(x))-r(\Psi)}{1+\epsilon r(C_{\Psi
,1-\gamma}^{c}(x))}\right\}  .
\end{align*}
The result follows from this because $C_{\Psi,\gamma}^{c}(x)\subset
C_{\Psi,1-\gamma}^{c}(x),$ so $r(C_{\Psi,\gamma}^{c}(x))\leq r(C_{\Psi
,1-\gamma}^{c}(x)),$ and (\ref{eq1}) is increasing in $r(A^{c}).\smallskip$
\end{proof}

It is interesting to consider the statistical meaning of the separate parts of
Proposition \ref{credregion} as the statements create a degree of ambiguity.
If a system of credible regions is being used, say $B_{\Psi,\gamma}(x),$ then
it makes sense to require that these sets are monotonically increasing in
$\gamma$ and the smallest set $\lim_{\gamma\searrow0}B_{\Psi,\gamma}(x)$
contains a single point which is taken as the estimate of $\psi.$ The size of
$B_{\Psi,\gamma}(x),$ for some specific $\gamma,$ can then be taken as an
assessment of the accuracy of the estimate where size is measured in some
application dependent way. The relative belief regions satisfy this and the
estimate, under the assumption of a unique maximizer of $RB_{\Psi}%
(\cdot\,|\,x),$ is $\psi(x).$ So effectively (i) is saying that $C_{\Psi
,\gamma}(x)$ is the most robust system of credible regions with respect to
posterior content. Note that we have to exclude sets $A$ with $\Pi_{\Psi
}\left(  A\,|\,x\right)  >\gamma^{\ast}(x)$ because, for example, the set
$A=\Psi$ is always optimally robust with respect to content but does not
provide a meaningful assessment of the accuracy of the estimate. Given that
$\psi(x)$ and the form of $C_{\Psi,\gamma}(x)$ are optimally robust, this
further supports the claim that relative belief estimation is optimally robust
to the choice of the marginal prior. Note that the sets in (ii) do not satisfy
the stated criteria for being a system of credible regions.

Part (iii) indicates that, when there are many sets with posterior content
exactly equal to $\gamma,$ and this is typically true in the continuous case,
then $C_{\Psi,\gamma}(x)$ is optimally robust among these sets with respect to
content. It makes sense to require $\gamma^{\ast}(x)\geq1/2$ for any credible
region as, if $\gamma^{\ast}(x)<1/2,$ then there is more belief that the true
value is in $C_{\Psi,\gamma}^{c}(x)$ than in $C_{\Psi,\gamma}(x).$

Applying Lemma 1 gives
\begin{align*}
\delta(C_{\Psi,\gamma}(x))  &  =\frac{\epsilon^{\ast}RB_{\Psi}(\psi(x)\pi
x)}{(1+\epsilon^{\ast}RB_{\Psi}(\psi(x)\,|\,x))}\times\\
&  \left\{  1-\frac{\Pi_{\Psi}\left(  C_{\Psi,\gamma}(x)\,|\,x\right)
}{RB_{\Psi}(\psi(x)\,|\,x)}\frac{1-\inf_{\psi\in C_{\Psi,\gamma}(x)}RB_{\Psi
}(\psi\,|\,x)}{(1+\epsilon^{\ast}\inf_{\psi\in C_{\Psi,\gamma}(x)}RB_{\Psi
}(\psi\,|\,x))}\right\}
\end{align*}
and this can be close to 1 when $RB_{\Psi}(\psi(x)\,|\,x)$ is large. So, while
$C_{\Psi,\gamma}(x)$ possesses an optimal robustness property with respect to
posterior content, this does not imply that the posterior content is
necessarily robust. This depends on other aspects of the particular problem
which will be discussed.

\section{Measuring Robustness Quantitatively}

To measure the robustness of an inference to the prior $\pi,$ when using the
$\epsilon$-contaminated class, it is natural to look at G\^{a}teaux
derivatives of the relevant quantity at $\pi$ in various directions $Q.$ The
derivative is a measure of the sensitivity of the inference to small changes
in the prior and so is local in nature. When the derivative is large for some
$Q,$ the inference is highly sensitive to the prior chosen and naturally this
is viewed negatively. In this section this behavior of relative belief
inferences is analyzed separately for $\epsilon$-contaminated classes for the
marginal $\pi_{\Psi}$ and the conditional $\pi(\cdot\,|\,\psi).$

\subsection{Sensitivity to the Marginal Prior\label{margprior}}

Consider the family of priors given by (\ref{prior1}) but now restricted to
those $Q$ that are also absolutely continuous with respect to $\nu_{\Psi}$ on
$\Psi$ and let $q$ denote the density of $Q.$ The posterior of $\psi$ based on
the contaminated prior is $\Pi_{\epsilon,\Psi}(\cdot\,|\,x)=(1-\epsilon
_{x})\Pi_{\Psi}(\cdot\,|\,x)+\epsilon_{x}Q(\cdot\,|\,x)$ where $\epsilon
_{x}=\epsilon m_{Q}\left(  x\right)  /[(1-\epsilon)m(x)+\epsilon m_{Q}\left(
x\right)  ],m_{Q}(x)=\int_{\Psi}m(x\,|\,\psi)\,Q(d\psi)$ and $Q(A\,|\,x)=\int
_{A}(m(x\,|\,\psi)/m_{Q}(x))\,Q(d\psi).$ The relative belief ratio for $\psi$
based on a general $\Pi_{\epsilon}$ equals $RB_{\epsilon,\Psi}(\psi
\,|\,x)=(1-\epsilon_{x})RB_{\Psi}(\psi\,|\,x)+\epsilon_{x}RB_{Q,\Psi}%
(\psi\,|\,x)$ and here, using (\ref{savagedickey}), $RB_{Q,\Psi}%
(\psi\,|\,x)=m(x\,|\,\psi)/m_{Q}\left(  x\right)  $ so
\begin{equation}
RB_{\epsilon,\Psi}(\psi\,|\,x)=\frac{RB_{\Psi}(\psi\,|\,x)}{1-\epsilon
(1-m_{Q}\left(  x\right)  /m(x))}. \label{robustrelbel}%
\end{equation}
The following result gives the G\^{a}teaux derivative of the relative belief ratio.

\begin{proposition}
\label{relbelmarg}The \textit{G\^{a}teaux derivative} of $RB_{\Psi}%
(\cdot\,|\,x)$ at $\psi$ in the direction $Q$ equals
\begin{equation}
RB_{\Psi}(\psi\,|\,x)\left\{  1-m_{Q}\left(  x\right)  /m\left(  x\right)
\right\}  . \label{gatrelbeltatio}%
\end{equation}

\end{proposition}

\begin{proof}
From (\ref{robustrelbel}),%
\[
\lim_{\epsilon\rightarrow0}\frac{RB_{\Psi}^{\epsilon}(\psi\,|\,x)-RB_{\Psi
}(\psi\,|\,x)}{\epsilon}=RB_{\Psi}(\psi\,|\,x)\lim_{\epsilon\rightarrow
0}\left\{  \frac{(1-m_{Q}\left(  x\right)  /m(x))}{1-\epsilon(1-m_{Q}\left(
x\right)  /m(x))}\right\}  .
\]
\smallskip
\end{proof}

\noindent The value of (\ref{gatrelbeltatio}) can be large simply because
$RB_{\Psi}(\psi\,|\,x)$ is large, so it makes more sense to look at the
relative change as given by $1-m_{Q}\left(  x\right)  /m\left(  x\right)  .$
Therefore, for small $\epsilon,$%
\[
\frac{\left\vert RB_{\epsilon,\Psi}(\psi\,|\,x)-RB_{\Psi}(\psi
\,|\,x)\right\vert }{RB_{\Psi}(\psi\,|\,x)}\approx\left\vert 1-\frac{m_{Q}%
(x)}{m(x)}\right\vert \epsilon
\]
implying a small relative change in $RB_{\Psi}(\psi\,|\,x)$ when
$m_{Q}(x)/m(x)$ is not large.

The G\^{a}teaux derivative of the strength of the evidence is now computed.

\begin{proposition}
The \textit{G\^{a}teaux} derivative of $\Pi_{\Psi}(RB_{\Psi}(\psi\,|\,x)\leq
RB_{\Psi}(\psi_{0}\,|\,x)\,|\,x)$ at $\psi_{0}$ in the direction $Q$ is
\[
\frac{m_{Q}\left(  x\right)  }{m(x)}\left\{
\begin{array}
[c]{c}%
Q(RB_{\Psi}(\psi\,|\,x)\leq RB_{\Psi}(\psi_{0}\,|\,x)\,|\,x)-\\
\Pi_{\Psi}(RB_{\Psi}(\psi\,|\,x)\leq RB_{\Psi}(\psi_{0}\,|\,x)\,|\,x)
\end{array}
\right\}  .
\]

\end{proposition}

\begin{proof}
The strength based on $\Pi_{\Psi}^{\epsilon}$ satisfies $\Pi_{\Psi}^{\epsilon
}(RB_{\epsilon,\Psi}(\psi\,|\,x)\leq RB_{\epsilon,\Psi}(\psi_{0}%
\,|\,x)\,|\,x)$\newline$=\left(  1-\epsilon_{x}\right)  \Pi_{\Psi
}(RB_{\epsilon,\Psi}(\psi\,|\,x)\leq RB_{\epsilon,\Psi}(\psi_{0}%
\,|\,x)\,|\,x)+\epsilon_{x}Q(RB_{\epsilon,\Psi}(\psi\,|\,x)\leq RB_{\epsilon
,\Psi}(\psi_{0}\,\newline|\,x)\,|\,x).$ So, using (\ref{robustrelbel}),%
\begin{align*}
\Pi_{\Psi}^{\epsilon}(RB_{\epsilon,\Psi}(\psi\,|\,x)  &  \leq RB_{\epsilon
,\Psi}(\psi_{0}\,|\,x)\,|\,x)=\Pi_{\Psi}(m(x\,|\,\psi)\leq m(x\,|\,\psi
_{0})\,|\,x)+\\
&  \epsilon_{x}\left\{  Q(m(x\,|\,\psi)\leq m(x\,|\,\psi_{0})\,|\,x)-\Pi
_{\Psi}(m(x\,|\,\psi)\leq m(x\,|\,\psi_{0})\,|\,x)\right\}  .
\end{align*}
This implies that
\begin{align*}
&  \lim_{\epsilon\rightarrow0}[\Pi_{\Psi}^{\epsilon}(RB_{\Psi}^{\epsilon}%
(\psi_{0}\,|\,x)\leq RB_{\Psi}^{\epsilon}(\psi_{0}\,|\,x)\,|\,x)-\Pi_{\Psi
}(RB_{\Psi}(\psi\,|\,x)\leq RB_{\Psi}(\psi_{0}\,|\,x)\,|\,x)]/\epsilon\\
&  =\frac{m_{Q}\left(  x\right)  }{m(x)}\left\{  Q(m(x\,|\,\psi)\leq
m(x\,|\,\psi_{0})\,|\,x)-\Pi_{\Psi}(m(x\,|\,\psi)\leq m(x\,|\,\psi
_{0})\,|\,x)\right\} \\
&  =\frac{m_{Q}\left(  x\right)  }{m(x)}\left\{
\begin{array}
[c]{c}%
Q(RB_{\Psi}(\psi\,|\,x)\leq RB_{\Psi}(\psi_{0}\,|\,x)\,|\,x)-\\
\Pi_{\Psi}(RB_{\Psi}(\psi\,|\,x)\leq RB_{\Psi}(\psi_{0}\,|\,x)\,|\,x)
\end{array}
\right\}  .
\end{align*}

\end{proof}

\noindent So the strength is robust to choice of the marginal prior $\pi
_{\Psi}$ whenever $m_{Q}(x)/m(x)$ is small.

For both the measure of evidence $RB_{\Psi}(\psi_{0}\,|\,x)$ and its strength,
the ratio $m_{Q}(x)/m(x)$ plays a key role in determining the robustness. The
implications of this are discussed in\ Section 5. Note that $\sup_{Q}%
m_{Q}(x)/m(x)=RB(\psi(x)\,|\,x)$ gives the worst case behavior of this ratio.

It is of interest to contrast these results with those for the commonly used
MAP inferences which are based on the posterior density $\pi_{\Psi}\left(
\cdot\,|\,x\right)  .$

\begin{proposition}
The \textit{G\^{a}teaux derivative of the posterior density of }$\psi$ in the
direction $Q$ at $\psi_{0}$\ is given by $\{m_{Q}(x)/m(x)\}\{q\left(  \psi
_{0}\,|\,x\right)  -\pi_{\Psi}\left(  \psi_{0}\,|\,x\right)  \}.$
\end{proposition}

\begin{proof}
Since $\pi_{\epsilon,\Psi}\left(  \psi\,|\,x\right)  =(1-\epsilon_{x}%
)\pi_{\Psi}(\psi\,|\,x)+\epsilon_{x}q(\psi|\,x)$ it follows that
\[
\lim_{\epsilon\rightarrow0}\frac{\pi_{\epsilon,\Psi}\left(  \psi
_{0}\,|\,x\right)  -\pi_{\Psi}\left(  \psi_{0}\,|\,x\right)  }{\epsilon}%
=\frac{m_{Q}\left(  x\right)  }{m(x)}(q\left(  \psi_{0}\,|\,x\right)
-\pi_{\Psi}\left(  \psi_{0}\,|\,x\right)  ).
\]

\end{proof}

\noindent Note that MAP-based inferences implicitly use $\pi_{\Psi}\left(
\psi_{0}\,|\,x\right)  $ as a measure of the evidence that $\psi_{0}$ is the
true value. Comparing this with the relative belief ratio we see that for
small $\epsilon,$%
\[
\frac{\left\vert \pi_{\epsilon,\Psi}\left(  \psi_{0}\,|\,x\right)  -\pi_{\Psi
}(\psi_{0}\,|\,x)\right\vert }{\pi_{\Psi}(\psi_{0}\,|\,x)}\approx\frac
{m_{Q}(x)}{m(x)}\left\vert 1-\frac{q\left(  \psi_{0}\,|\,x\right)  }{\pi
_{\Psi}\left(  \psi_{0}\,|\,x\right)  }\right\vert \epsilon
\]
and the relative change in $\pi_{\Psi}\left(  \psi_{0}\,|\,x\right)  $ is
dependent on the ratio of the posteriors as well as $m_{Q}(x)/m(x).$ So if
$\pi_{\Psi}\left(  \psi_{0}\,|\,x\right)  $ is small relative to $q\left(
\psi_{0}\,|\,x\right)  $ we will get a big relative change and this suggests
that MAP inferences are much less robust than relative belief inferences. A
similar result is obtained for the Bayesian p-value in Evans and Zou (2001).

\subsection{Sensitivity to\ the Conditional Prior}

Consider now priors for\ $\theta$ of the form $\Pi_{\epsilon}=[(1-\epsilon
)\Pi(\cdot\,|\,\psi)+\epsilon Q(\cdot\,|\,\psi)]\times\Pi_{\Psi}$ where
$Q(\cdot\,|\,\psi)$ is a probability measure on $\Psi^{-1}\{\psi\}$ absolutely
continuous with respect to $\nu_{\Psi^{-1}\{\psi\}}$ with density
$q(\cdot\,|\,\psi),$ for each $\psi\in\Psi.$ So the marginal prior of $\psi$
is now fixed and the conditional prior of $\theta$ is perturbed. The posterior
of $\psi$ based on this prior is $\Pi_{\epsilon,\Psi}\left(  \cdot
\,|\,x\right)  =\left(  1-\epsilon_{x}\right)  \Pi_{\Psi}\left(
\cdot\,|\,x\right)  +\epsilon_{x}Q_{\Psi}\left(  \cdot\,|\,x\right)  $ where
$Q_{\Psi}\left(  A\,|\,x\right)  =\int_{A}(m_{Q}(x\,|\,\psi)/m_{Q}%
(x))\,\Pi_{\Psi}(d\psi),m_{Q}(x\,|\,\psi)=\int_{\Psi^{-1}\{\psi\}}f_{\theta
}(x)\,Q(d\theta\,|\,\psi)$ and $m_{Q}(x)=\int_{\Psi}m_{Q}(x\,|\,\psi
)\,\Pi_{\Psi}(d\psi).$

The relative belief ratio for $\psi$ based on $\Pi_{\epsilon}$ equals
$RB_{\epsilon,\Psi}(\psi\,|\,x)=(1-\epsilon_{x})RB_{\Psi}(\psi\,|\,x)+\epsilon
_{x}RB_{Q,\Psi}(\psi\,|\,x)$ where now $B_{Q,\Psi}(\psi_{0}\,|\,x)=m_{Q}%
(x\,|\,\psi_{0})/m_{Q}(x).$ This leads to the following result.

\begin{proposition}
\label{relbelcond}The G\^{a}teaux derivative of $RB_{\Psi}(\cdot\,|\,x)$ at
$\psi_{0}$ in the direction $Q$ is $\{m_{Q}(x)/m\left(  x\right)
\}(RB_{Q,\Psi}(\psi_{0}\,|\,x)-RB_{\Psi}(\psi_{0}\,|\,x)).$
\end{proposition}

\begin{proof}
Clearly,%
\[
\lim_{\epsilon\rightarrow0}\frac{RB_{\epsilon,\Psi}(\psi_{0}\,|\,x)-RB_{\Psi
}(\psi_{0}\,|\,x)}{\epsilon}=\frac{m_{Q}\left(  x\right)  }{m(x)}(RB_{Q,\Psi
}(\psi_{0}\,|\,x)-RB_{\Psi}(\psi_{0}\,|\,x)).\smallskip
\]

\end{proof}

\noindent The implications of this result for robustness are discussed in
Section 5.

Now consider the robustness of the strength of the evidence.

\begin{proposition}
\label{condpriorstrength}If $RB_{\Psi}(\cdot\,|\,x)$ has a discrete
distribution with support containing no limit points, the \textit{G\^{a}teaux}
derivative of $\Pi_{\Psi}(RB_{\Psi}(\psi\,|\,x)\leq RB_{\Psi}(\psi
_{0}\,|\,x)\,|\,x)$ at $\psi_{0}$ in the direction $Q$ equals 0. When
$RB_{\Psi}(\cdot\,|\,x)$ has a continuous distribution under $\Pi_{\Psi}%
(\cdot\,|\,x)$ with density $g(\cdot\,|\,x),$ the \textit{G\^{a}teaux}
derivative of $\Pi_{\Psi}(RB_{\Psi}(\psi\,|\,x)\leq RB_{\Psi}(\psi
_{0}\,|\,x)\,|\,x)$ at $\psi_{0}$ in the direction $Q$ equals%
\[
\{m_{Q}(x)/m\left(  x\right)  )\}RB_{Q,\Psi}(\psi_{0}\,|\,x)g(RB_{\Psi}%
(\psi_{0}\,|\,x)\,|\,x).
\]

\end{proposition}

\begin{proof}
Since%
\begin{align*}
&  \Pi_{\Psi}(RB_{\epsilon,\Psi}(\psi\,|\,x)\leq RB_{\epsilon,\Psi}(\psi
_{0}\,|\,x)\,|\,x)\\
&  =\Pi_{\Psi}\left(
\begin{array}
[c]{c}%
(1-\epsilon_{x})RB_{\Psi}(\psi\,|\,x)+\epsilon_{x}RB_{Q,\Psi}(\psi\,|\,x)\\
\leq(1-\epsilon_{x})RB_{\Psi}(\psi_{0}\,|\,x)+\epsilon_{x}RB_{Q,\Psi}(\psi
_{0}\,|\,x)
\end{array}
\,|\,x\right)  ,
\end{align*}
then, for all $\epsilon>0$ such that $\epsilon_{x}\leq1,$%
\begin{align*}
&  \Pi_{\Psi}(RB_{\epsilon,\Psi}(\psi\,|\,x)\leq RB_{\epsilon,\Psi}(\psi
_{0}\,|\,x)\,|\,x)\\
&  \leq\Pi_{\Psi}\left(  RB_{\Psi}(\psi\,|\,x)\leq RB_{\Psi}(\psi
_{0}\,|\,x)+\frac{\epsilon_{x}}{1-\epsilon_{x}}RB_{Q,\Psi}(\psi_{0}%
\,|\,x)\,|\,x\right)
\end{align*}
and for all $\epsilon<0,$%
\begin{align*}
&  \Pi_{\Psi}(RB_{\epsilon,\Psi}(\psi\,|\,x)\leq RB_{\epsilon,\Psi}(\psi
_{0}\,|\,x)\,|\,x)\\
&  \geq\Pi_{\Psi}\left(  RB_{\Psi}(\psi\,|\,x)\leq RB_{\Psi}(\psi
_{0}\,|\,x)+\frac{\epsilon_{x}}{1-\epsilon_{x}}RB_{Q,\Psi}(\psi_{0}%
\,|\,x)\,|\,x\right)  .
\end{align*}
When $RB_{\Psi}(\cdot\,|\,x)$ has a discrete distribution with support
containing no limit points, then the lower and upper bounds equal $\Pi_{\Psi
}(RB_{\Psi}(\psi\,|\,x)\leq RB_{\Psi}(\psi_{0}\,|\,x)\,|\,x)$ for all
$\epsilon$ small enough and the result follows. When $RB_{\Psi}(\cdot\,|\,x)$
has a continuous distribution with density $g(\cdot\,|\,x),$ then%
\begin{align*}
&  \lim_{\epsilon\rightarrow0}\frac{\Pi_{\Psi}(RB_{\epsilon,\Psi}(\psi
_{0}\,|\,x)\leq RB_{\epsilon,\Psi}(\psi_{0}\,|\,x)\,|\,x)-\Pi_{\Psi}(RB_{\Psi
}(\psi\,|\,x)\leq RB_{\Psi}(\psi_{0}\,|\,x)\,|\,x)}{\epsilon}\\
&  =\{m_{Q}\left(  x\right)  /m(x)\}RB_{Q,\Psi}(\psi_{0}\,|\,x)g(RB_{\Psi
}(\psi_{0}\,|\,x)\,|\,x).
\end{align*}

\end{proof}

\noindent From this it is seen that in the discrete case the strength is
insensitive to local changes in the prior.

Consider the continuous case. When there is strong evidence either for or
against $\psi_{0},$ then $RB_{\Psi}(\psi_{0}\,|\,x)$ will be in the right or
left tail correspondingly of the posterior distribution of $RB_{\Psi}%
(\cdot\,|\,x)$ and so $g(RB_{\Psi}(\psi_{0}\,|\,x)\,|\,x)$ will tend to be
small. As such the strength will be robust to small changes in the prior
provided $m_{Q}\left(  x\right)  /m(x)$ is not large. When there is not strong
evidence however, then $g(RB_{\Psi}(\psi_{0}\,|\,x)\,|\,x)$ could be large
and, if $m_{Q}\left(  x\right)  /m(x)$ is not small, then the strength is not
robust. This underscores a recommendation in Baskurt and Evans (2013) that in
the continuous case the parameter be discretized when assessing the evidence
and its strength. For this, when $\psi$ is real-valued, let $\delta>0$ be the
difference between two $\psi$ values that is deemed to be of practical
importance. The prior and posterior distributions of $\psi$ discretized to the
intervals $[\psi_{0}+(2i-1)\delta/2,\psi_{0}+(2i+1)\delta/2)$ for $i\in%
\mathbb{Z}
$ are then used to assess the hypothesis corresponds to the interval
$[\psi_{0}-\delta/2,\psi_{0}+\delta/2).$ By Proposition
\ref{condpriorstrength} the strength is then insensitive to small changes in
the prior.

It is perhaps not surprising that the robustness behavior of the relative
belief ratio and its strength is more complicated when considering the effect
of the conditional prior than with the marginal prior. The optimality results
concerning robustness to the marginal prior underscore this.

\section{Robustness and Prior-Data Conflict}

The existence of a prior-data conflict means that the data support certain
values of $\psi=\Psi(\theta)$ being the true value but the prior places little
or no mass there. While various measures can be used to determine whether or
not such a conflict has occurred, a logical approach is based on the
factorization of the joint probability measure for $(\theta,x)$ given by
$\Pi\times P_{\theta}=\Pi(\cdot\,|\,T)\times M_{T}\times P(\cdot\,|\,T),$
where $T$ is a minimal sufficient statistic, $\Pi(\cdot\,|\,T)$ is the
posterior probability measure for $\theta,$ $M_{T}$ is the prior predictive
probability measure of $T$ and $P(\cdot\,|\,T)$ is the conditional probability
measure of the data given $T.$ The measure $P(\cdot\,|\,T)$ is then available
for computing probabilities relevant to checking the model $\{f_{\theta
}:\theta\in\Theta\},$ the measure $M_{T}$ is available for computing
probabilities relevant to checking the prior and $\Pi(\cdot\,|\,T)$ is the
relevant probability measure for computing probabilities for $\theta.$ A
statistical analysis then proceeds by checking the model, perhaps via a tail
probability based on a discrepancy statistic, and then proceeding to check the
prior if the data does not contradict the model. If both the model and prior
are not contradicted by the data, then we can proceed to inference about
$\theta.$ The logic behind this sequence lies in part with the fact that it
makes no sense to check a prior if the model fails. Furthermore, separating
the check of the prior from that of the model provides more information in the
event of a conflict arising, as it is then possible to identify where the
failure lies, namely, with the model or with the prior.

In Evans and Moshonov (2006) this factorization was adhered to and the tail
probability
\begin{equation}
M_{T}(m_{T}(t)\leq m_{T}(T(x))) \label{priorcon1}%
\end{equation}
was advocated for checking the prior where $m_{T}$ is the density of $M_{T}$
with respect to some support measure. So if (\ref{priorcon1}) is small, then
the observed value $T(x)$ of the minimal sufficient statistic lies in the
tails of $M_{T}$ and there is an indication of a prior-data conflict. In Evans
and Jang (2011a) the validity of this approach was firmly established by the
proof that (\ref{priorcon1}) converges to $\Pi(\pi(\theta)\leq\pi
(\theta_{true}))$ under i.i.d. sampling and some additional weak conditions.
Furthermore, it was shown how to modify (\ref{priorcon1}) so as to achieve
invariance under choice of the minimal sufficient statistic. Also, Evans and
Moshonov (2006) argued that (\ref{priorcon1}) should be replaced by
$M_{T}(m_{T}(t)\leq m_{T}(T(x))\,|\,U(T(x)))$ for any maximal ancillary $U(T)$
as the variation in $T$ due to to $U(T)$ has nothing to do with $\theta$ and
so reflects nothing about the prior. The tail probability (\ref{priorcon1}) is
a check on the full prior and Evans and Moshonov (2006) also developed methods
for checking factors of the prior so a failure in the prior could be isolated
to a particular aspect.

First, however, consider the case when $\Psi(\theta)=\theta$ and interest is
in the robustness of inferences to the whole prior. From the results in
Section \ref{margprior}, it is seen that the ratio $m_{Q}(x)/m(x)=m_{Q,T}%
(T(x))/m_{T}(T(x)),$ where $m_{Q,T}(T(x))=\int_{\Theta}f_{\theta
,T}(T(x))\,Q(d\theta),$ plays a key role in determining the local sensitivity
in the direction given by $Q,$ of the inferences for given observed data $x.$
This depends on $Q$ and the worst case is given by
\begin{equation}
\sup_{Q}\frac{m_{Q,T}(T(x))}{m_{T}(T(x))}=\sup_{Q}\frac{\int_{\Theta}%
f_{\theta,T}(T(x))\,Q(d\theta)}{m_{T}(T(x))}=RB(\theta(x)\,|\,x)
\label{priorcon2}%
\end{equation}
and note that $\theta(x)$ is the MLE in this case as well as the relative
belief estimate. Notice that when (\ref{priorcon1}) is small, so there is an
indication of a prior-data conflict existing, then $m_{T}(T(x))$ is relatively
small when compared to other values of $m_{T}(t)$ which are not influenced by
the data. This implies that the prior is having a big influence relative to
the data and so a lack of robustness can be expected.

This phenomenon is well-illustrated in the following examples where
ancillaries play no role because of Basu's theorem.\smallskip

\noindent\textbf{Example 1.} \textit{Location normal model.}

Suppose that $x=(x_{1},\ldots,x_{n})$ is a sample from the $N(\mu,1)$
distribution with $\mu\sim N(\mu_{0},\sigma_{0}^{2}).$ Then $M_{T}$ is given
by $T(x)=\bar{x}\sim N(\mu_{0},1/n+\sigma_{0}^{2}).$ When $Q$ is the
$N(\mu_{1},\sigma_{1}^{2})$ distribution, then $M_{Q,T}$ is given by $\bar
{x}\sim N(\mu_{1},1/n+\sigma_{1}^{2}).$ This implies that
\begin{equation}
\frac{m_{Q,T}(T(x))}{m_{T}(T(x))}=\sqrt{\frac{1/n+\sigma_{0}^{2}}%
{1/n+\sigma_{1}^{2}}}\exp\left\{  -\frac{1}{2}\left[
\begin{array}
[c]{c}%
\left(  1/n+\sigma_{1}^{2}\right)  ^{-1}\left(  \bar{x}-\mu_{1}\right)
^{2}-\\
\left(  1/n+\sigma_{0}^{2}\right)  ^{-1}\left(  \bar{x}-\mu_{0}\right)  ^{2}%
\end{array}
\right]  \right\}  \label{ratiolocnorm}%
\end{equation}
and, as a function of $(\mu_{1},\sigma_{1}^{2})$ this is maximized when
$\mu_{1}=\bar{x},\sigma_{1}^{2}=0.$ Notice that this supremum converges to
$\infty$ as $\bar{x}\rightarrow\pm\infty$ and such values correspond to
prior-data conflict with respect to the $N(\mu_{0},\sigma_{0}^{2})$ prior.

Now, consider a numerical example. A sample of size $n=20$ was generated from
the $N(0,1)$ distribution obtaining $\bar{x}=0.2591.$ When the base prior is
$N(0.5,1)$ then (\ref{priorcon1}) equals $0.8141$ and accordingly there is no
indication of any prior-data conflict. Also, $\sup_{Q}\left(  m_{Q}%
(x)/m(x)\right)  =4.7109$ which seems modest as it describes the worst case
robustness behavior. In Table 1 some values of $m_{Q}(x)/m(x)$ are recorded
when $Q$ is a $N(\mu_{1},\sigma_{1}^{2})$ distribution for various values of
$\mu_{1}$ and $\sigma_{1}^{2}$ as these might be expected to be realistic
directions in which to perturb the base prior. In all cases the value of
$m_{Q}(x)/m(x)$ is quite modest and the maximum value of (\ref{ratiolocnorm})
is $1.0534.$ Overall it can be concluded here that the analysis is robust to
local perturbations of the prior.%

\begin{table}[tbp] \centering
\begin{tabular}
[c]{|r|r|r|r|r|r|}\hline
$\mu_{1}$ & $\sigma_{1}^{2}$ & $m_{Q}(x)/m(x)$ & $\mu_{1}$ & $\sigma_{1}^{2}$
& $m_{Q}(x)/m(x)$\\\hline
$-3.0$ & $1$ & $0.0065$ & $0.5$ & $0.5$ & $1.3474$\\\hline
$-2.0$ & $1$ & $0.0905$ & $0.5$ & $1.0$ & $1.0000$\\\hline
$-1.0$ & $1$ & $0.4832$ & $0.5$ & $2.0$ & $0.7254$\\\hline
$1.0$ & $1$ & $0.7917$ & $0.5$ & $3.0$ & $0.5975$\\\hline
$2.0$ & $1$ & $0.2428$ & $0.5$ & $50.0$ & $0.1488$\\\hline
$3.0$ & $1$ & $0.0287$ & $0.5$ & $100.0$ & $0.1053$\\\hline
\end{tabular}
\caption{The ratio $m_Q(x)/m(x)$ in Example 1 when there is no conflict.}%
\end{table}%

Now consider an example where there is prior-data conflict. In this case a
sample of $n=20$ is generated from a $N(4,1)$ distribution obtaining $\bar
{x}=4.0867$ and the same base prior is used. The value of (\ref{priorcon1}) is
$0.0005$ and so there is a strong indication of prior-data conflict.
Furthermore, $\sup_{Q}\left(  m_{Q}(x)/m(x)\right)  =2096.85$ which certainly
indicates a lack of robustness. In Table 2 some values of $m_{Q}(x)/m(x)$ are
recorded when $Q$ is a $N(\mu_{1},\sigma_{1}^{2})$ distribution for various
values of $\mu_{1}$ and $\sigma_{1}^{2}$. It is seen that the value of
$m_{Q}(x)/m(x)$ can be relatively large and the maximum value of
(\ref{ratiolocnorm}) is $468.86.$ So it can be concluded that the analysis
based on the model, prior and observed data, will not be robust to local
perturbations of the prior when there is prior-data conflict. $\blacksquare
$\smallskip%

\begin{table}[tbp] \centering
\begin{tabular}
[c]{|r|r|r|r|r|r|}\hline
$\mu_{1}$ & $\sigma_{1}^{2}$ & $m_{Q}(x)/m(x)$ & $\mu_{1}$ & $\sigma_{1}^{2}$
& $m_{Q}(x)/m(x)$\\\hline
$-3.0$ & $1$ & $1.88\times10^{-8}$ & $0.5$ & $0.5$ & $0.0053$\\\hline
$-2.0$ & $1$ & $9.97\times10^{-8}$ & $0.5$ & $1.0$ & $1.0000$\\\hline
$-1.0$ & $1$ & $2.00\times10^{-4}$ & $0.5$ & $2.0$ & $14.2070$\\\hline
$1.0$ & $1$ & $4.90\times10^{-1}$ & $0.5$ & $3.0$ & $32.5842$\\\hline
$2.0$ & $1$ & $5.75\times10^{1}$ & $0.5$ & $50.0$ & $58.2823$\\\hline
$3.0$ & $1$ & $2.61\times10^{2}$ & $0.5$ & $100.0$ & $43.9565$\\\hline
\end{tabular}
\caption{The ratio $m_Q(x)/m(x)$ in Example 1 when there is conflict.}%
\end{table}%

\noindent\textbf{Example 2.} \textit{Bernoulli model.}

Suppose that $x=(x_{1},\ldots,x_{n})$ is a sample from a Bernoulli$(\theta)$
and the prior is $\theta\sim\,$beta$(\alpha_{0},\beta_{0})$ for some choice of
$(\alpha_{0},\beta_{0}).$ A minimal sufficient statistic is $T(x)=\sum
_{i=1}^{n}x_{i}\sim\hbox{Binomial}(n,\theta)$ and then
\[
m_{T}\left(  t\right)  =\binom{n}{t}\frac{\Gamma(\alpha_{0}+\beta_{0})}%
{\Gamma(\alpha_{0})\Gamma(\beta_{0})}\frac{\Gamma(t+\alpha_{0})\Gamma
(n-t+\beta_{0})}{\Gamma(n+\alpha_{0}+\beta_{0})}.
\]
Also,
\[
\sup_{Q}\left(  m_{Q}(x)/m(x)\right)  =\frac{\Gamma(\alpha_{0})\Gamma
(\beta_{0})}{\Gamma(\alpha_{0}+\beta_{0})}\frac{\Gamma(n+\alpha_{0}+\beta
_{0})}{\Gamma(t+\alpha_{0})\Gamma(n-t+\beta_{0})}\bar{x}^{n\bar{x}}(1-\bar
{x})^{n(1-\bar{x})}.
\]

To illustrate the relationship between prior-data conflict and robustness,
consider a numerical example. Suppose that $\alpha_{0}=5\ $and $\beta_{0}=20.$
Generating a sample of size $n=20$ from the Bernoulli$(0.25)$ gave the value
$n\bar{x}=3.$ In this case (\ref{priorcon1}) equals $0.7100$ and there is no
indication of any prior-data conflict. Also, $\sup_{Q}\left(  m_{Q}%
(x)/m(x)\right)  =1.4211$ which indicates that the inferences will be
generally robust to small deviations. If $m_{Q}(x)/m(x)$ is computed for
various $Q,$ where $Q$ is a beta$(\alpha_{1},\beta_{1}),$ then in all cases it
is readily seen that this ratio is quite reasonable in value as indeed it is
bounded above by $1.4211.$

A sample of $n=20$ was also generated from a Bernoulli$(0.9)$ with the same
prior being used. In this case $n\bar{x}=17$ and (\ref{priorcon1}) equals
$6.2\times10^{-6},$ so there is a strong indication of prior-data conflict.
Also, $\sup_{Q}\left(  m_{Q}(x)/m(x)\right)  =46396.43$ which indicates that
the inferences will be generally not be robust to small deviations. Table 3
provides some values of $m_{Q}(x)/m(x)$ for $Q$ given by a beta$(\alpha
_{1},\beta_{1})$ for various choices of $(\alpha_{1},\beta_{1})$ and there are
several large values. $\blacksquare$\smallskip%

\begin{table}[tbp] \centering
\begin{tabular}
[c]{|r|r|r|r|r|r|}\hline
$\alpha_{1}$ & $\beta_{1}$ & $m_{Q}(x)/m(x)$ & $\alpha_{1}$ & $\beta_{1}$ &
$m_{Q}(x)/m(x)$\\\hline
$20$ & $5$ & $32647.89$ & $5$ & $1$ & $21523.28$\\\hline
$15$ & $5$ & $25729.50$ & $5$ & $25$ & $0.12$\\\hline
$10$ & $5$ & $15010.95$ & $5$ & $22$ & $0.41$\\\hline
$5$ & $5$ & $3996.37$ & $5$ & $20$ & $1.00$\\\hline
$1$ & $5$ & $125.87$ & $5$ & $16$ & $6.77$\\\hline
\end{tabular}
\caption{The ratio $m_Q(x)/m(x)$ in Example 2 when there is conflict.}%
\end{table}%

Now consider the case when $\theta=(\theta_{1},\theta_{2})\in\Theta_{1}%
\times\Theta_{2}$ so the prior factors as $\pi(\theta)=\pi_{2}(\theta
_{2}\,|\,\theta_{1})\pi_{1}(\theta_{1}).$ Presumably the conditional prior
$\pi_{2}(\cdot\,|\,\theta_{1})$ and the marginal prior\ $\pi_{1}$ are elicited
and the goal is inference about some $\psi=\Psi(\theta).$ It is then
preferable to check the prior by checking each individual component for
prior-data conflict as this leads to more information about where a conflict
exists when it does.

In general, it is not clear how to check the individual components but in
certain contexts a particular structure holds that allows for this. Suppose
that all ancillaries are independent of the minimal sufficient statistic and
so can be ignored. The more general situation is covered in Evans and Moshonov (2006).

As discussed in Evans and Moshonov (2006), suppose there is a statistic $V(T)$
such that the marginal distribution of $V(T)$ is dependent only on $\theta
_{1}.$ Such a statistic is referred to as being ancillary for $\theta_{2}$
given $\theta_{1}.$ Naturally we want $V(T)$ to be a maximal ancillary for
$\theta_{2}$ given $\theta_{1}.$ An appropriate tail probability for checking
$\pi_{1}$ is then given by%
\begin{equation}
M_{V(T)}(m_{V(T)}(v)\leq m_{V(T)}(V(T(x)))), \label{priorcon4}%
\end{equation}
as $M_{V(T)}$ does not depend on $\pi_{2}(\cdot\,|\,\theta_{1}).$ A natural
order is to check $\pi_{1}$ first and then check $\pi_{2}(\cdot\,|\,\theta
_{1})$ for prior-data conflict, whenever no prior-data conflict is found for
$\pi_{1}.$ The appropriate tail probability for checking $\pi_{2}%
(\cdot\,|\,\theta_{1})$ is given by
\begin{equation}
M_{T}(m_{T}(t\,|\,V(T(x)))\leq m_{T}(T(x)\,|\,V(T(x)))\,|\,V(T(x))).
\label{priorcon3}%
\end{equation}
Note that this is assessing whether or not $\pi_{2}(\cdot\,|\,\theta_{1})$ is
a suitable prior for $\theta_{2}$ among those $\theta_{1}$ values deemed to be
suitable according to the prior $\pi_{1}.$ If (\ref{priorcon3}) were to be
used before (\ref{priorcon4}), then it would not be possible to assess if a
failure was due to where $\pi_{1}$ was placing the bulk of its mass or was
caused by where the conditional priors were placing their mass. Notice that%
\begin{equation}
\frac{m_{Q,T}(T(x))}{m_{T}(T(x))}=\frac{m_{Q,T}(T(x)\,|\,V(T(x)))}%
{m_{T}(T(x)\,|\,V(T(x)))}\frac{m_{Q,V(T)}(V(T(x)))}{m_{V(T)}(V(T(x)))},
\label{priorcon5}%
\end{equation}
so prior-data conflict with either $\pi_{1}$ or $\pi_{2}(\cdot\,|\,\theta
_{1})$ could lead to large values of the ratio on the left for certain choices
of $Q.$ When only the conditional prior of $\theta_{2}$ given $\theta_{1}$ is
perturbed, then $m_{Q,V(T)}(V(T(x)))=$ $m_{V(T)}(V(T(x))).$

Letting $f_{\theta_{1},V}$ denote the density of $V,$ then%
\begin{align*}
&  \frac{m_{Q,V(T)}(V(T(x)))}{m_{V(T)}(V(T(x)))}\\
&  =\int_{\Theta_{1}}\frac{f_{\theta_{1},V}(V(T(x)))}{m_{V(T)}(V(T(x)))}%
\,Q_{1}(d\theta_{1})\leq RB_{1}(\theta_{1}(V(T(x)))\,|\,V(T(x)))
\end{align*}
where $RB_{1}(\cdot\,|\,V(T(x)))$ gives the relative belief ratios for
$\theta_{1}$ based on having observed $V(T(x)).$ The right-hand side gives the
worst-case behavior of the second factor in (\ref{priorcon5}).

Now consider the robustness of relative belief inferences for a general
$\psi=\Psi(\theta).$ The following result generalizes Propositions
\ref{relbelmarg} and \ref{relbelcond} as we consider a general perturbation to
the prior, namely, $\Pi_{\epsilon}=(1-\epsilon)\Pi+\epsilon Q$ and the proof
is the same as that of Proposition \ref{relbelcond}.

\begin{proposition}
The G\^{a}teaux derivative of $RB_{\Psi}(\cdot\,|\,x)$ at $\psi$ in the
direction $Q$ is $\{m_{Q}(x)/m\left(  x\right)  \}(RB_{Q,\Psi}(\psi
\,|\,x)-RB_{\Psi}(\psi\,|\,x)).$
\end{proposition}

\noindent The factor $RB_{Q,\Psi}(\psi\,|\,x)-RB_{\Psi}(\psi\,|\,x)$ can be
big simply because we choose a prior $Q$ that is very different than $\Pi.$
For example, $RB_{\Psi}(\psi\,|\,x)$ may be big (small) because there is
considerable evidence in favor of (against) $\psi$ being the true value and we
can choose a prior $Q$ that doesn't (does) place mass near $\psi.$ As such, it
makes sense to standardize the derivative by dividing by this factor and this
leaves the robustness determined again by $m_{Q}(x)/m\left(  x\right)  .$

Suppose now that $Q$ and $\Pi$ have the same marginal for $\varsigma
=\Xi(\theta).$ Then, $m_{Q}(x)=m\left(  x\right)  \int_{\Xi}\int_{\Xi
^{-1}\{\tau\}}RB(\theta\,|\,x)\,Q(d\theta\,|\,\psi)\,\Pi_{\Xi}(d\varsigma)\leq
m\left(  x\right)  \int_{\Xi}RB(\theta_{\varsigma}(x)\,|\,x)\,$\newline%
$\Pi_{\Xi}(d\varsigma)$ where $\theta_{\varsigma}(x)=\arg\sup\{RB(\theta
)\,|\,x):\Xi(\theta)=\varsigma\}.$ Therefore,
\begin{equation}
\frac{m_{Q}(x)}{m\left(  x\right)  }\leq\int_{\Xi}RB(\theta_{\varsigma
}(x)\,|\,x)\,\Pi_{\Xi}(d\varsigma) \label{priorcon6}%
\end{equation}
and the the right-hand side gives the worst-case behavior of the first factor
in (\ref{priorcon5}) when $\Xi(\theta)=\theta_{1}$ which is related to
prior-data conflict with the prior on $\theta_{2}.$

The following is a standard example where priors are specified
hierarchically.\smallskip

\noindent\textbf{Example 3} \textit{Location-scale normal model.}

Suppose that $x=(x_{1},\ldots,x_{n})$ is a sample from the $N(\mu,\sigma^{2})$
distribution with $\mu\,|\,\sigma^{2}\sim N(\mu_{0},\tau_{0}^{2}\sigma
^{2}),\sigma^{-2}\sim\,$gamma$_{rate}(\alpha_{0},\beta_{0}).$ Then
$T(x)=(\bar{x},||x-\bar{x}1||^{2})$ is a minimal sufficient statistic for the
model. Note that the prior is chosen by eliciting values for $\mu_{0},\tau
_{0}^{2},\alpha_{0},\beta_{0}$ and so there is interest in how sensitive
inferences are to perturbations in each component separately. The posterior
distribution of $(\mu,\sigma^{2})$ is given by $\mu\,|\,\sigma^{2},T(x)\sim
N(\mu_{x},\left(  n+1/\tau_{0}^{2}\right)  ^{-1}\sigma^{2}),\sigma
^{-2}\,|\,T(x)\sim\,$gamma$_{rate}\left(  \alpha_{0}+n/2,\beta(\bar{x}%
,s^{2})\right)  $ where $\mu_{x}=(n+1/\tau_{0}^{2})^{-1}(n\bar{x}+\mu_{0}%
/\tau_{0}^{2})$ and $\beta(\bar{x},s^{2})$\newline$=\beta_{0}+(n-1)s^{2}%
/2+n(\bar{x}-\mu_{0})^{2}/2(n\tau_{0}^{2}+1)$ with $s^{2}=||x-\bar{x}%
1||^{2}/(n-1).$

Consider first inferences for $\psi=\Psi(\theta)=\sigma^{2}$ and note that
$V(T(x))=||x-\bar{x}1||^{2}$ is ancillary given $\psi$ and its distribution
depends on $\psi.$ Therefore, the prior on $\sigma^{2}$ is checked first using
the prior predictive for $V(T(x)).$ An easy calculation gives that the prior
distribution of $s^{2}=V(T(x))/(n-1)$ is $(\beta_{0}/\alpha_{0})F(n-1,2\alpha
_{0})$ and this specifies (\ref{priorcon4}). While the results of Section 4.1
apply here, consider the behavior of the relative belief ratio $RB_{1}%
(\sigma^{2}\,|\,V(T(x)))$ which is based on only observing $V(T(x))$ rather
than $T(x).$ By Proposition \ref{relbelmarg} this has G\^{a}teaux derivative
depending on $m_{Q,V(T)}(V(T(x)))/m_{V(T)}(V(T(x))).$ Notice, however, that
relative belief ratios accumulate evidence in a simple way. For any statistic
$V(T(x)),$ then
\[
RB_{\Psi}(\psi\,|\,T(x))=\frac{\pi_{\Psi}(\psi\,|\,T(x))}{\pi_{\Psi}(\psi
)}=\frac{\pi_{\Psi}(\psi\,|\,V(T(x)))}{\pi_{\Psi}(\psi)}\frac{\pi_{\Psi}%
(\psi\,|\,T(x))}{\pi_{\Psi}(\psi\,|\,V(T(x)))}%
\]
where the first factor gives the evidence obtained after observing $V(T(x))$
and the second factor gives the evidence obtained after observing $T(x)$
having already observed $V(T(x)).$ So $RB_{1}(\sigma^{2}\,|\,x)=RB_{1}%
(\sigma^{2}\,|\,V(T(x)))[RB_{1}(\sigma^{2}\,|\,T(x))/$\newline$RB_{1}%
(\sigma^{2}\,|\,V(T(x)))]$ with the same interpretation for the factors. As
such, a lack of robustness of $RB_{1}(\sigma^{2}\,|\,V(T(x))),$ which can be
connected to prior-data conflict through (\ref{priorcon4}), implies a lack of
robustness for $RB_{1}(\sigma^{2}\,|\,x).$

When no prior-data conflict is obtained for the prior on $\sigma^{2},$ then it
makes sense to look for prior-data conflict with the prior on $\mu$ which is
typically the parameter of primary interest. So now consider perturbations to
the prior on $\mu$ and the relationship to prior-data conflict with this
prior. The conditional distribution of $T(x)$ given $V(T(x))$ is given by the
conditional prior predictive of $\bar{x}$ given $s^{2}$ which is distributed
as $\mu_{0}+\tilde{\sigma}t_{n+2\alpha_{0}-1}$ where $\tilde{\sigma}%
^{2}=\left\{  \tau_{0}^{2}\left(  n\tau_{0}^{2}+1\right)  \left(  2\beta
_{0}+(n-1)s^{2}\right)  +1\right\}  /\{n\tau_{0}^{2}\left(  n+2\alpha
_{0}-1\right)  \}$ specifying (\ref{priorcon3}). Furthermore, for
(\ref{priorcon6}), $\Xi(\theta)=\sigma^{2}$ and $\theta_{\sigma^{2}}%
(x)=(\bar{x},\sigma^{2})$ with%
\begin{align*}
&  RB((\bar{x},\sigma^{2})\,|\,x)\\
&  =\frac{(n\tau_{0}^{2}+1)^{\frac{1}{2}}}{\beta_{0}^{\alpha_{0}}}\frac
{\Gamma(\alpha_{0})}{\Gamma(\alpha_{0}+n/2)}(\beta(\bar{x},s^{2}))^{\alpha
_{0}+\frac{n}{2}}\left(  \frac{1}{\sigma^{2}}\right)  ^{\frac{n}{2}}%
\exp\left\{  -\frac{(n-1)s^{2}}{2\sigma^{2}}\right\}
\end{align*}
and so%
\[
\int_{0}^{\infty}RB((\bar{x},\sigma^{2})\,|\,x)\,\Pi_{1}(d\sigma^{-2}%
)=(n\tau_{0}^{2}+1)^{\frac{1}{2}}\left(  \frac{\beta(\bar{x},s^{2})}{\beta
_{0}+(n-1)s^{2}/2}\right)  ^{\alpha_{0}+\frac{n}{2}}.
\]

Now consider a number of numerical examples where the base prior is always
specified by $\mu_{0}=0,\tau_{0}^{2}=1,\alpha_{0}=5$ and $\beta_{0}=5.$ The
behavior of the two factors in (\ref{priorcon5}) is examined when there is no
prior-data conflict and when there is.

A sample of size $n=20$ was generated from the $N(0,1)$ distribution obtaining
$\bar{x}=-0.1066,s^{2}=0.9087.$ So there should be no prior-data conflict with
the prior on $\sigma^{2}.$ Indeed, (\ref{priorcon4}) equals $0.7626$ so there
is no indication of any problems with the prior on $\sigma^{2}.$ Values of
$m_{Q,V(T)}(V(T(x)))/m_{V(T)}(V(T(x)))$ are recorded in Table 4 when the
marginal prior on $\sigma^{2}$ is perturbed by a gamma$_{rate}(\alpha
_{1},\beta_{1})$ distribution for various values of $\alpha_{1}$ and
$\beta_{1}$. In all cases, the ratio is small and indicates robustness to
local perturbations of the prior on $\sigma^{2}.$ Note that the worst case
behavior, over all possible directions, is given by the maximized relative
belief ratio for $\sigma^{2}$ based on $V(T(x))$ which occurs at $\sigma
^{2}=s^{2}$ and equals%
\begin{align*}
&  RB_{1}(s^{2}\,|\,V(T(x)))\\
&  =\frac{\Gamma(\alpha_{0})}{\Gamma(\alpha_{0}+(n-1)/2)}\beta_{0}%
^{-\alpha_{0}}e^{-\frac{n-1}{2}}(s^{2})^{-\frac{n-1}{2}}\left(  \frac
{(n-1)s^{2}}{2}+\beta_{0}\right)  ^{\frac{n-1}{2}+\alpha_{0}}.
\end{align*}
In this case $RB_{1}(s^{2}\,|\,V(T(x)))=1.7479.$%

\begin{table}[tbp] \centering
\begin{tabular}
[c]{|c|c|c|c|c|c|}\hline
$\alpha_{1}$ & $\beta_{1}$ & $\frac{m_{Q,V(T)}(V(T(x)))}{m_{V(T)}(V(T(x)))}$ &
$\alpha_{1}$ & $\beta_{1}$ & $\frac{m_{Q,V(T)}(V(T(x)))}{m_{V(T)}(V(T(x)))}%
$\\\hline
$5$ & $1$ & $0.05$ & $1$ & $5$ & $0.07$\\\hline
$5$ & $2$ & $0.38$ & $2$ & $5$ & $0.25$\\\hline
$5$ & $4$ & $0.99$ & $4$ & $5$ & $0.81$\\\hline
$5$ & $10$ & $0.34$ & $10$ & $5$ & $0.53$\\\hline
\end{tabular}
\caption{The ratio $m_{Q,V(T)}(V(T(x)))/m_{V(T)}(V(T(x)))$ in Example 3 when there is no conflict with the prior on $\sigma^2$.}%
\end{table}%

Next a sample of size $n=20$ from the $N(0,25)$ was generated obtaining
$\bar{x}=0.0950,s^{2}=23.9593.$ So there is clearly prior-data conflict with
the prior on $\sigma^{2}.$ This is reflected in the value of (\ref{priorcon4})
which equals $0.64\times10^{-5}.$ Table 5 shows that there is a serious lack
of robustness. The worst case behavior is given by $RB_{1}(s^{2}%
\,|\,V(T(x)))=40484.68.$%

\begin{table}[tbp] \centering
\begin{tabular}
[c]{|r|r|r|r|r|r|}\hline
$\alpha_{1}$ & $\beta_{1}$ & $\frac{m_{Q,V(T)}(V(T(x)))}{m_{V(T)}(V(T(x)))}$ &
$\alpha_{1}$ & $\beta_{1}$ & $\frac{m_{Q,V(T)}(V(T(x)))}{m_{V(T)}(V(T(x)))}%
$\\\hline
\multicolumn{1}{|c|}{$5$} & \multicolumn{1}{|c|}{$1$} &
\multicolumn{1}{|c|}{$0.00$} & \multicolumn{1}{|c|}{$1$} &
\multicolumn{1}{|c|}{$5$} & \multicolumn{1}{|c|}{$5517.42$}\\\hline
\multicolumn{1}{|c|}{$5$} & \multicolumn{1}{|c|}{$2$} &
\multicolumn{1}{|c|}{$0.01$} & \multicolumn{1}{|c|}{$2$} &
\multicolumn{1}{|c|}{$5$} & \multicolumn{1}{|c|}{$1245.26$}\\\hline
\multicolumn{1}{|c|}{$5$} & \multicolumn{1}{|c|}{$4$} &
\multicolumn{1}{|c|}{$2.34$} & \multicolumn{1}{|c|}{$4$} &
\multicolumn{1}{|c|}{$5$} & \multicolumn{1}{|c|}{$13.78$}\\\hline
\multicolumn{1}{|c|}{$5$} & \multicolumn{1}{|c|}{$10$} &
\multicolumn{1}{|c|}{$23.51$} & \multicolumn{1}{|c|}{$10$} &
\multicolumn{1}{|c|}{$5$} & \multicolumn{1}{|c|}{$0.00$}\\\hline
\end{tabular}
\caption{The ratio $m_{Q,V(T)}(V(T(x)))/m_{V(T)}(V(T(x)))$ in Example 3 when there is conflict with the prior on $\sigma^2$.}%
\end{table}%

It is also relevant to consider what happens concerning the robustness of
inferences about $\sigma^{2}$ when there is prior-data conflict with the prior
on $\mu$ but not with the prior on $\sigma^{2}.$ A sample of $n=20$ was
generated from the $N(10,1)$ distribution obtaining $\bar{x}=9.7041,s^{2}%
=1.0082,$ so there is clearly prior-data conflict with the prior on $\mu$ but
not with the prior on $\sigma^{2}.$ The value of (\ref{priorcon4}) equals
$0.6460$ which gives no reason to doubt the relevance of the prior on
$\sigma^{2}.$ Table 6 shows that $m_{Q,V(T)}(V(T(x)))/m_{V(T)}((T(x)))$ is
small and indicates robustness to local perturbations of the prior on
$\sigma^{2}.$ The worst case behavior is given by $RB_{1}(s^{2}%
\,|\,V(T(x)))=1.7218.$ This reinforces the claim that the tail probabilities
(\ref{priorcon4}) and (\ref{priorcon3}) are measuring different aspects of the
data conflicting with the prior.%

\begin{table}[tbp] \centering
\begin{tabular}
[c]{|r|r|r|r|r|r|}\hline
$\alpha_{1}$ & $\beta_{1}$ & $\frac{m_{Q,V(T)}(V(T(x))}{m_{V(T)}(V(T(x))}$ &
$\alpha_{1}$ & $\beta_{1}$ & $\frac{m_{Q,V(T)}(V(T(x))}{m_{V(T)}(V(T(x))}%
$\\\hline
\multicolumn{1}{|c|}{$5$} & \multicolumn{1}{|c|}{$1$} &
\multicolumn{1}{|c|}{$0.03$} & \multicolumn{1}{|c|}{$1$} &
\multicolumn{1}{|c|}{$5$} & \multicolumn{1}{|c|}{$0.09$}\\\hline
\multicolumn{1}{|c|}{$5$} & \multicolumn{1}{|c|}{$2$} &
\multicolumn{1}{|c|}{$0.29$} & \multicolumn{1}{|c|}{$2$} &
\multicolumn{1}{|c|}{$5$} & \multicolumn{1}{|c|}{$0.31$}\\\hline
\multicolumn{1}{|c|}{$5$} & \multicolumn{1}{|c|}{$4$} &
\multicolumn{1}{|c|}{$0.92$} & \multicolumn{1}{|c|}{$4$} &
\multicolumn{1}{|c|}{$5$} & \multicolumn{1}{|c|}{$0.86$}\\\hline
\multicolumn{1}{|c|}{$5$} & \multicolumn{1}{|c|}{$10$} &
\multicolumn{1}{|c|}{$0.44$} & \multicolumn{1}{|c|}{$10$} &
\multicolumn{1}{|c|}{$5$} & \multicolumn{1}{|c|}{$0.38$}\\\hline
\end{tabular}
\caption{The ratio $m_{Q,V(T)}(V(T(x)))/m_{V(T)}(V(T(x)))$ in Example 3 when there is conflict with the prior on $\mu$ but not with the prior on $\sigma^2$.}%
\end{table}%

Now consider perturbations to the prior on $\mu$ with the prior on $\sigma
^{2}$ fixed. A sample of $n=20$ was generated from a $N(0,1)$ obtaining
$\bar{x}=-0.1066,s^{2}=0.9087$ so there is clearly no prior-data conflict with
either component. This is reflected in the value of (\ref{priorcon3}) which
equals $0.9150.$ Table 7 shows that the first factor $m_{Q,T}%
(T(x)\,|\,V(T(x)))/m_{T}(T(x)\,|\,V(T(x)))$ in (\ref{priorcon5}) is small when
the conditional prior on $\mu$ is perturbed by $N(\mu_{1},\tau_{1}^{2})$
priors and thus demonstrates robustness to perturbations in these directions.
The worst case behavior is given by $\int_{0}^{\infty}RB((\bar{x},\sigma
^{2})\,|\,x)\,\Pi_{1}(d\sigma^{-2})=4.6099$ which is comparatively small.%

\begin{table}[tbp] \centering
\begin{tabular}
[c]{|r|r|r|r|r|r|}\hline
$\mu_{1}$ & $\tau_{1}^{2}$ & $\frac{m_{Q,T}(T(x)\,|\,V(T(x)))}{m_{T}%
(T(x)\,|\,V(T(x)))}$ & $\mu_{1}$ & $\tau_{1}^{2}$ & $\frac{m_{Q,T}%
(T(x)\,|\,V(T(x)))}{m_{T}(T(x)\,|\,V(T(x)))}$\\\hline
\multicolumn{1}{|c|}{$-2$} & \multicolumn{1}{|c|}{$1$} &
\multicolumn{1}{|c|}{$0.17$} & \multicolumn{1}{|c|}{$0$} &
\multicolumn{1}{|c|}{$2$} & \multicolumn{1}{|c|}{$0.51$}\\\hline
\multicolumn{1}{|c|}{$-1$} & \multicolumn{1}{|c|}{$1$} &
\multicolumn{1}{|c|}{$0.66$} & \multicolumn{1}{|c|}{$0$} &
\multicolumn{1}{|c|}{$3$} & \multicolumn{1}{|c|}{$0.34$}\\\hline
\multicolumn{1}{|c|}{$1$} & \multicolumn{1}{|c|}{$1$} &
\multicolumn{1}{|c|}{$0.54$} & \multicolumn{1}{|c|}{$0$} &
\multicolumn{1}{|c|}{$4$} & \multicolumn{1}{|c|}{$0.26$}\\\hline
\multicolumn{1}{|c|}{$2$} & \multicolumn{1}{|c|}{$1$} &
\multicolumn{1}{|c|}{$0.12$} & \multicolumn{1}{|c|}{$0$} &
\multicolumn{1}{|c|}{$5$} & \multicolumn{1}{|c|}{$0.21$}\\\hline
\end{tabular}
\caption{The ratio $m_{Q,T}(T(x)\,|\,V(T(x))))/m_{T}(T(x)\,|\,V(T(x)))$ in Example 3 when there is no conflict with the prior on $\sigma^2$ or with the prior on $\mu$.}%
\end{table}%

Table 8 gives some values of $m_{Q,T}(T(x)\,|\,V(T(x)))/m_{T}%
(T(x)\,|\,V(T(x)))$ when a sample of $n=20$ was generated from a $N(0,25),$
obtaining $\bar{x}=0.0950,s^{2}=23.9593.$ So in this case there is prior-data
conflict with the prior on $\sigma^{2}$ but not with the prior on $\mu.$ The
value of (\ref{priorcon3}) equals $0.9150$ which gives no indication of
prior-data conflict with the prior on $\mu.$ The tabulated values also
indicate no serious robustness concerns as does $\int_{0}^{\infty}RB((\bar
{x},\sigma^{2})\,|\,x)\,\Pi_{1}(d\sigma^{-2})=4.5838.$ This also reinforces
the claim that the tail probabilities (\ref{priorcon4}) and (\ref{priorcon3})
are measuring different aspects of the data conflicting with the prior.%

\begin{table}[tbp] \centering
\begin{tabular}
[c]{|r|r|r|r|r|r|}\hline
$\mu_{1}$ & $\tau_{1}^{2}$ & $\frac{m_{Q,T}(T(x)\,|\,V(T(x)))}{m_{T}%
(T(x)\,|\,V(T(x)))}$ & $\mu_{1}$ & $\tau_{1}^{2}$ & $\frac{m_{Q,T}%
(T(x)\,|\,V(T(x)))}{m_{T}(T(x)\,|\,V(T(x)))}$\\\hline
\multicolumn{1}{|c|}{$-2$} & \multicolumn{1}{|c|}{$1$} &
\multicolumn{1}{|c|}{$0.87$} & \multicolumn{1}{|c|}{$0$} &
\multicolumn{1}{|c|}{$2$} & \multicolumn{1}{|c|}{$0.51$}\\\hline
\multicolumn{1}{|c|}{$-1$} & \multicolumn{1}{|c|}{$1$} &
\multicolumn{1}{|c|}{$0.96$} & \multicolumn{1}{|c|}{$0$} &
\multicolumn{1}{|c|}{$3$} & \multicolumn{1}{|c|}{$0.34$}\\\hline
\multicolumn{1}{|c|}{$1$} & \multicolumn{1}{|c|}{$1$} &
\multicolumn{1}{|c|}{$0.98$} & \multicolumn{1}{|c|}{$0$} &
\multicolumn{1}{|c|}{$4$} & \multicolumn{1}{|c|}{$0.26$}\\\hline
\multicolumn{1}{|c|}{$2$} & \multicolumn{1}{|c|}{$1$} &
\multicolumn{1}{|c|}{$0.90$} & \multicolumn{1}{|c|}{$0$} &
\multicolumn{1}{|c|}{$5$} & \multicolumn{1}{|c|}{$0.21$}\\\hline
\end{tabular}
\caption{The ratio $m_{Q,T}(T(x)\,|\,V(T(x))))/m_{T}(T(x)\,|\,V(T(x)))$ in Example 3 when there is conflict with the prior on $\sigma^2$ but not with the prior on $\mu$.}%
\end{table}%

Table 9 gives some values of $m_{Q,T}(T(x)\,|\,V(T(x)))/m_{T}%
(T(x)\,|\,V(T(x)))$ when a sample of $n=20$ was generated from a $N(10,1)$
obtaining $\bar{x}=9.7941,s^{2}=1.0082.$ So in this case there is prior-data
conflict with the prior on $\mu$ but not with the prior on $\sigma^{2}.$ The
value of (\ref{priorcon3}) equals $0.1691\times10^{-9}$ which gives a clear
indication of prior-data conflict with the prior on $\mu.$ In this case the
tabulated values indicate a clear lack of robustness with respect to the prior
on $\mu.$ Also, $\int_{0}^{\infty}RB((\bar{x},\sigma^{2})\,|\,x)\,\Pi
_{1}(d\sigma^{-2})=8,046,933,962$ indicates that the worst case behavior with
respect to robustness is terrible.%

\begin{table}[tbp] \centering
\begin{tabular}
[c]{|r|r|r|r|r|r|}\hline
$\mu_{1}$ & $\tau_{1}^{2}$ & $\frac{m_{Q,T}(T(x)\,|\,V(T(x)))}{m_{T}%
(T(x)\,|\,V(T(x)))}$ & $\mu_{1}$ & $\tau_{1}^{2}$ & $\frac{m_{Q,T}%
(T(x)\,|\,V(T(x)))}{m_{T}(T(x)\,|\,V(T(x)))}$\\\hline
\multicolumn{1}{|c|}{$-2$} & \multicolumn{1}{|c|}{$1$} &
\multicolumn{1}{|c|}{$0.01$} & \multicolumn{1}{|c|}{$0$} &
\multicolumn{1}{|c|}{$2$} & \multicolumn{1}{|c|}{$117,584$}\\\hline
\multicolumn{1}{|c|}{$-1$} & \multicolumn{1}{|c|}{$1$} &
\multicolumn{1}{|c|}{$0.10$} & \multicolumn{1}{|c|}{$0$} &
\multicolumn{1}{|c|}{$3$} & \multicolumn{1}{|c|}{$5,611,980$}\\\hline
\multicolumn{1}{|c|}{$1$} & \multicolumn{1}{|c|}{$1$} &
\multicolumn{1}{|c|}{$10.83$} & \multicolumn{1}{|c|}{$0$} &
\multicolumn{1}{|c|}{$4$} & \multicolumn{1}{|c|}{$26,012,609$}\\\hline
\multicolumn{1}{|c|}{$2$} & \multicolumn{1}{|c|}{$1$} &
\multicolumn{1}{|c|}{$132.09$} & \multicolumn{1}{|c|}{$0$} &
\multicolumn{1}{|c|}{$5$} & \multicolumn{1}{|c|}{$55,478,630$}\\\hline
\end{tabular}
\caption{The ratio $m_{Q,T}(T(x)\,|\,V(T(x))))/m_{T}(T(x)\,|\,V(T(x)))$ in Example 3 when there is no conflict with the prior on $\sigma^2$ but there is with the prior on $\mu$.}%
\end{table}%

\section{Conclusions}

Several optimal robustness results have been derived here for relative belief
inferences. These and other results suggest a natural preference for these
inferences over other Bayesian inferences for estimation and hypothesis
assessment. Even though relative belief inferences may be the most robust to
choice of prior, this does not guarantee that they are robust in practice. The
issue of practical robustness in a given problem is seen to be connected with
whether or not there is prior-data conflict. With no prior-data conflict the
inferences are robust to small changes in the prior, at least in the sense
measured here. This adds support to the point-of-view that checking for
prior-data conflict is an essential aspect of good statistical practice.

It is interesting that the worst case behavior of the measure of sensitivity
is associated with the maximized value of a relative belief ratio. The actual
maximum value attained is meaningless, however, as there is no way to
calibrate this as opposed to calibrating the relative belief ratio at a fixed
value via the strength. The relative belief estimate is consistent, however,
and the relative belief ratio at this value will, at least in the continuous
case, converge to infinity. So large values would seem to be associated with
high evidence in favor. What has been shown here is that large values can be
associated with prior-data conflict and a lack of robustness rather than
providing high evidence. When prior-data conflict is encountered the prior can
be modified, following Evans and Jang (2011b), to avoid this. While objections
can be raised to taking such a step, it seems necessary if we want to report a
valid characterization of the evidence obtained.

In Baskurt and Evans (2013) the relationship between relative belief ratios
and Bayes factors is examined. Both serve as measures of evidence but the
relative belief ratio is a simpler, more direct measure and it has many nice
mathematical properties. It is the case too that often relative belief ratios
and Bayes factors agree. For example, in the case of continuous priors, when
the Bayes factor at a point is defined as a limit, then these quantities are
the same. As such, it is reasonable to expect that the results derived here
about relative belief inferences will apply equally well to inferences based
on Bayes factors.

\section{References}

\noindent Baskurt, Z. and Evans, M. (2013) Hypothesis assessment and
inequalities for Bayes factors and relative belief ratios. Bayesian Analysis,
8, 3, 569-590.\smallskip\ 

\noindent Berger, J. O. (1990) Robust Bayesian analysis: sensitivity to the
prior. Journal of Statistical Planning and Inference, 25, 303-328.\smallskip

\noindent Berger, J. O. (1994) An overview of robust Bayesian analysis (with
discussion). Test, 3, 5--124.\smallskip

\noindent de la Horra, J. and Fernandez, C. (1994) Bayesian analysis under
$\epsilon$-contaminated priors: a trade-off between robustness and precision.
Journal of Statistical Planning and Inference, 38, 13-30.\smallskip

\noindent Dey, D. K. and Birmiwall, L. R. (1994) Robust Bayesian analysis
using divergence measures. Statistics and Probability Letters, 20,
287-294.\smallskip

\noindent Evans, M. and Moshonov, H. (2006) Checking for prior-data conflict.
Bayesian Analysis, 1, 4, 893-914.\smallskip

\noindent Evans, M., Guttman, I. and Swartz, T. (2006) Optimality and
computations for relative surprise inferences. Canadian Journal of Statistics,
34, 1, 113-129.\smallskip

\noindent Evans, M. and Jang G. H. (2011a) A limit result for the prior
predictive. Statistics and Probability Letters, 81, 1034-1038.\smallskip

\noindent Evans, M. and Jang G. H. (2011b) Weak informativity and the
information in one prior relative to another. Statistical Science, 26, 3,
423-439.\smallskip

\noindent Evans, M. and Jang G. H. (2011c) Inferences from prior-based loss
functions. arXiv:1104.3258 [math.ST].\smallskip\ 

\noindent Evans, M. and Shakhatreh, M. (2008) Optimal properties of some
Bayesian inferences. Electronic Journal of Statistics, 2, 1268-1280.\smallskip

\noindent Evans, M. and Zou, T. (2001) Robustness of relative surprise
inferences to choice of prior. Recent Advances in Statistical Methods,
Proceedings of Statistics 2001 Canada: The 4th Conference in Applied
Statistics Montreal, Canada 6 - 8 July 2001, Yogendra P. Chaubey (ed.),
90-115, Imperial College Press.\smallskip

\noindent Huber, P. J. (1973) The use of Choquet capacities in statistics.
Bulletin of the International Statistical Institute, 45, 181-191.\smallskip

\noindent Rios Insua, D. and Ruggeri, F. (2000) Robust Bayesian Analysis.
Springer-Verlag.\smallskip

\noindent Rudin, W. (1974) Real and Complex Analysis, Second Edition.
McGraw-Hill, New York.\vspace{2pt}

\noindent Ruggeri, F. and Wasserman, L. (1993) Infinitesimal sensitivity of
posterior distributions. The Canadian Journal of Statistics, 21, 2,
195-203.\smallskip

\noindent Wasserman, L. (1989) A robust Bayesian interpretation of likelihood
regions. Annals of Statistics, 17,3, 1387-1393.

\section*{Appendix}

\paragraph{Proof of Lemma \ref{huber1}}

Note first that%
\begin{equation}
Q\left(  A\,|\,x\right)  =\int_{A}\frac{m(x\,|\,\psi)}{m_{q}(x)}\,Q(d\psi).
\label{Qpost}%
\end{equation}
Therefore, using (\ref{Qpost}),%
\begin{align*}
\Pi_{\Psi}^{\epsilon}\left(  A\,|\,x\right)   &  =\frac{\left(  1-\epsilon
\right)  m(x)\Pi_{\Psi}\left(  A\,|\,x\right)  +\epsilon m_{q}(x)\,Q\left(
A\,|\,x\right)  }{\left(  1-\epsilon\right)  m(x)+\epsilon m_{q}(x)}\\
&  =\frac{\Pi_{\Psi}\left(  A\,|\,x\right)  +\frac{\epsilon}{\left(
1-\epsilon\right)  m(x)}\int_{A}m(x\,|\,\psi)\,Q(d\psi)}{1+\frac{\epsilon
}{\left(  1-\epsilon\right)  m(x)}\left\{  \int_{A}m(x\,|\,\psi)\,Q(d\psi
)+\int_{A^{c}}m(x\,|\,\psi)\,Q(d\psi)\right\}  }\\
&  \leq\frac{\Pi_{\Psi}\left(  A\,|\,x\right)  +\frac{\epsilon}{\left(
1-\epsilon\right)  m(x)}\int_{A}m(x\,|\,\psi)\,Q(d\psi)}{1+\frac{\epsilon
}{\left(  1-\epsilon\right)  m(x)}\int_{A}m(x\,|\,\psi)\,Q(d\psi)}%
\end{align*}
and the last inequality is an equality when $Q(A^{c})=0.$ Result (i) then
follows since $(p+y)/(1+y)=1-(1-p)/(1+y)$ is increasing in $y\geq0$ when
$1-p>0$ and clearly $\sup_{Q}\int_{A}m(x\,|\,\psi)\,Q(d\psi)=\lim
_{\delta\downarrow0}\int_{A}m(x\,|\,\psi)\,Q_{\delta}(d\psi)=\sup_{\psi\in
A}m(x\,|\,\psi)$ where $Q_{\delta}$ places all of its mass on the set
$\{\psi:m(x\,|\,\psi)\geq\sup_{\psi\in A}m(x\,|\,\psi)-\delta\}\cap A.$

For result (ii) we have that
\begin{align*}
\Pi_{\Psi}^{\epsilon}\left(  A\,|\,x\right)   &  \geq\frac{\Pi_{\Psi}\left(
A\,|\,x\right)  +\frac{\epsilon}{\left(  1-\epsilon\right)  m(x)}\int
_{A}m(x\,|\,\psi)\,Q(d\psi)}{1+\frac{\epsilon}{\left(  1-\epsilon\right)
m(x)}\left\{  \int_{A}m(x\,|\,\psi)\,Q(d\psi)+\sup_{\psi\in A^{c}}%
m(x\,|\,\psi)\right\}  }\\
&  \geq\frac{\Pi_{\Psi}\left(  A\,|\,x\right)  }{1+\frac{\epsilon}{\left(
1-\epsilon\right)  m(x)}\sup_{\psi\in A^{c}}m(x\,|\,\psi)}%
\end{align*}
where the first inequality is obvious and the second follows since
$(p+y)/(1+y+b)=1-(1-p+b)/(1+y+b)$ is increasing in $y\geq0$ when $1-p+b>0$ and
so the minimum is attained at $y=0.$ The inequalities are equalities whenever
$Q(A)=0.$ We then argue as in (i).

For (iii) a direct calculation yields
\begin{align*}
\delta(A)  &  =\Pi_{\Psi}^{upper}\left(  A\,|\,x\right)  -\Pi_{\Psi}%
^{lower}\left(  A\,|\,x\right) \\
&  =\frac{\Pi_{\Psi}\left(  A\,|\,x\right)  (1+\epsilon^{\ast}r(A^{c}%
))+\epsilon^{\ast}r(A)(1+\epsilon^{\ast}r(A^{c}))-\Pi_{\Psi}\left(
A\,|\,x\right)  (1+\epsilon^{\ast}r(A))}{(1+\epsilon^{\ast}r(A))(1+\epsilon
^{\ast}r(A^{c}))}\\
&  =\frac{\Pi_{\Psi}\left(  A\,|\,x\right)  \epsilon^{\ast}(r(A^{c}%
)-r(A))}{(1+\epsilon^{\ast}r(A))(1+\epsilon^{\ast}r(A^{c}))}+\frac
{\epsilon^{\ast}r(A)}{1+\epsilon^{\ast}r(A)}.
\end{align*}

Result (iv) follows from $\delta(A^{c})=\sup_{Q}(1-\Pi_{\Psi}^{\epsilon
}\left(  A\,|\,x\right)  )-\inf_{Q}(1-\Pi_{\Psi}^{\epsilon}\left(
A\,|\,x\right)  )=\sup_{Q}(-\Pi_{\Psi}^{\epsilon}\left(  A\,|\,x\right)
)-\inf_{Q}(-\Pi_{\Psi}^{\epsilon}\left(  A\,|\,x\right)  )=\delta(A).$

\end{document}